%% file: Main.tex
\documentclass[a4paper,12pt]{amsart}
\usepackage[cp1251]{inputenc}
\usepackage[T2A]{fontenc}
\usepackage[russian,ukrainian,english]{babel}
\usepackage{amscd}
\usepackage{amssymb}
\usepackage{amsmath,amsthm,amsfonts}
\usepackage{hhline}

\usepackage{cite}

\newlength\deltaH
\newlength\deltaHcenter
\newlength\deltaV
\newlength\deltaVcenter
\newtheorem{theorem}{Theorem}[section]
\newtheorem{lem}{Lemma}[section]
\newtheorem*{thm}{Theorem}
\newtheorem{ozn}{Definition}[section]
\newtheorem{remk}{Remark}[section]
\newtheorem{nas}{Corollary}[section]
\newtheorem{prop}{Proposition}[section]

\newcommand{\rr}{\mathbb{R}}
\newcommand{\nn}{\mathbb{N}}
\renewcommand{\emptyset}{\varnothing}
\newcommand{\Cl}[1]{\overline{#1}}
\newcommand{\Int}[1]{\mathrm{Int}\,{#1}}
\newcommand{\Fr}[1]{\mathrm{Fr}\,{#1}}

\newcommand{\pr}{\mathop\mathrm{pr}\nolimits}
\newcommand{\restrict}[2]{#1\raisebox{-0.65ex}{$\left|\vphantom{{#1}_{#2}}\right.$}_{#2}}

\newcommand{\dist}{\mathop\mathrm{dist}\nolimits}
\newcommand{\df}{\mathop\mathrm{dist}_{\mathcal{F}}\nolimits}
\newcommand{\cM}{\mathcal{M}(\rr^{2})}
\newcommand{\diam}{\mathop\mathrm{diam}\nolimits}

\usepackage{graphicx}

\begin{document}

\begin{center}
{\large\bf On continuous functions on two-dimensional disk which are regular in its interior points.}
\end{center}
\vspace{3mm}
\begin{center}
{\small Yevgen Polulyakh}%
\footnote{
Institute of mathematics, Tereschenkivska str. 3, 01601, Kyiv, Ukraine
\par
e-mail: {\ttfamily polulyah@imath.kiev.ua}
}
\end{center}
\begin{center}
{\small Institute of mathematics of Natl. Acad. Sci. of Ukraine, Kyiv}
\end{center}
\vspace{5mm}

{\noindent \small\textbf{Abstract.}
We introduce a class of regular continuous functions on the closed 2-disk and show that each function from this class is topologically conjugate to a linear function defined on a sqare, a closed half-disk or a closed disk.
}\medskip


\input{Intro.tex}
\input{regular_on_interior.tex}

\input{level_sets.tex}
\input{foliation_rectification.tex}

\input{biblio.tex}

\end{document}

%% file: Intro.tex
\section*{Introduction}

Let us remind that a pseudo-harmonic function (see~\cite{MrJ}) on a closed domain in the plain (see~\cite{Yu}) is a continuous function such that in a small open neighbourhood of each interior point of the domain it is topologically conjugate to a function $\mathop{Re} z^n + const$ in a open neighbourhood of zero for a certain $n \in \nn$.

In the process of investigation of pseudo-harmonic functions the following problem arises.
Suppose we have two pseudo-harmonic functions $f$ and $g$ on a same closed domain $\Cl{G}$ and all ``singularities'' of $f$ and $g$ are contained in some ``convenient'' subsets $R(f)$ and $R(g)$ of $\Cl{G}$ which are unions of some families of connected components of level sets of $f$ and $g$ respectively with $\Fr{\Cl{G}}$. Suppose that we have a homeomorphism $\Phi_0 : R(f) \rightarrow R(g)$ which complies with the relation $g \circ \Phi_0 = f$.

The question is whether we can extend $\Phi_0$ to a homeomorphism $\Phi : \Cl{G} \rightarrow \Cl{G}$, such that $g \circ \Phi = f$.

In order to construct an extension we have to find a homeomorphism $\Phi_U : \Cl{U} \rightarrow \Cl{V}$ such that $g \circ \Phi_U = f$ and $\Phi_U |_{R(f)} = \Phi_0$ for every connected component $U$ of the set $\Cl{G} \setminus R(f)$ and a connected component $V$ of $\Cl{G} \setminus R(g)$ which is bounded by the set $\Fr{V} = \Phi_0(\Fr{U}) = \Phi_0(\Cl{U} \cap R(f))$.

When a domain $\Cl{G}$ is bounded by a finite number of simple closed curves and functions $f$ and $g$ have a finite number of ``singular'' points both in $G$ and on the frontier $\Fr{G}$ one could verify that for a ``convenient'' subset $R(f)$ a closure $\Cl{U}$ of every component $U$ of the set $\Cl{G} \setminus R(f)$ is homeomorphic to a closed 2-disk (see~\cite{Morse, MrJ}) and $f$ is ``regular'' in a sense on $\Cl{U}$.

In this article we will discuss the definition and different properties of regular functions on a closed 2-disk the most remarkable of which is given by the following statement.

\begin{thm}
Let $f$ and $g$ are regular functions on a closed 2-disk $D$.

Every homeomorphism $\varphi_0 : \partial D \rightarrow \partial D$ of the frontier $\partial D$ of $D$ which complies with the equality $g \circ \varphi_0 = f$ can be extended to a homeomorphism $\varphi : D \rightarrow D$ which satisfies the equality $g \circ \varphi = f$.
\end{thm}

%% file: regular_on_interior.tex
\section{Weakly regular functions on disk and their properties.}\label{subsection_regular_on_interior}

Let $W$ be a domain in the plane $\rr^{2}$, $f : \Cl{W} \rightarrow \rr$ be a continuous function.
We denote
\[
D^{2} = \{ z \;|\; |z| \leq 1 \} \,, \quad
D_{+}^{2} = \{ z \;|\; |z| < 1 \mbox{ and } Im z \geq 0 \} \,.
\]

\begin{ozn}\label{ozn_f_1}
We call $z_{0} \in W$ a \emph{regular point} of the function $f$ if there exist an open neighbourhood $U \subseteq W$ of $z_0$ and a homeomorphism $\varphi : U \rightarrow \Int{D^{2}}$ such that $\varphi(z_{0}) = 0$ and $f \circ \varphi^{-1}(z) = Re z + f(z_{0})$ for all $z \in \Int{D^{2}}$.

$U$ is called a \emph{canonical neighbourhood} of $z_0$.
\end{ozn}

\begin{ozn}\label{ozn_f_2}
Call $z_{0} \in \Fr{W}$ a \emph{regular boundary point} of $f$ if there exist an open neighbourhood $U$ in the space $\Cl{W}$ and a homeomorphism $\psi : U \rightarrow D_{+}^{2}$ such that $\psi(z_{0}) = 0$, $\psi(U \cap f^{-1}(f(z_{0}))) = \{0\} \times [0, 1)$, $\psi(U \cap \Fr{W}) = (-1, 1) \times \{0\}$ and a function $f \circ \psi^{-1}$ is strictly monotone on the interval $(-1, 1) \times \{0\}$.

A neighbourhood $U$ is called \emph{canonical}.
\end{ozn}

\begin{remk}\label{remk_kanon_neighb}
It is easy to see that canonical neighbourhood in definitions~\ref{ozn_f_1} and~\ref{ozn_f_2} can be chosen arbitrarily small.
\end{remk}

Let $D$ is a closed subset of the plane which is homeomorphic to $D^2$. Let us fix a bypass direction of a boundary circle $\Fr{D}$.

Assume that when we bypass the circle $\Fr{D}$ in the positive direction we consecutively pass through points $z_1, \ldots, z_{2n-1}, z_{2n}$ for some $n \geq 2$, and also not necessarily $z_k \neq z_{k+1}$. For every $k \in \{1, \ldots, 2n\}$ we designate by $\gamma_k$ an arc of the circle $\Fr{D}$ which originates in $z_k$ and ends in either $z_{k+1}$ when $k < 2n$ or $z_1$ if $k = 2n$, so that the movement direction along it coinsides with the bypass direction of $\Fr{D}$. Write $\mathring{\gamma}_k = \gamma_k \setminus \{z_k, z_{k+1}\}$ when $k \in \{1, \ldots, 2n-1\}$, $\mathring{\gamma}_{2n} = \gamma_{2n} \setminus \{z_{2n}, z_1\}$.

\begin{ozn}\label{ozn_weak_regular}
Assume that for a continuous function $f : D \rightarrow \rr$ there exist such $n = \mathcal{N}(f) \geq 2$ and a sequence of points $z_1, \ldots, z_{2n-1}, z_{2n} \in \Fr{D}$ (which are passed through in this order when the circle $\Fr{D}$ is bypassed in the positive direction) that following properties are fulfilled:
\begin{itemize}
	\item [1)] every point of a domain $\Int{D} = D \setminus \Fr{D}$ is a regular point of $f$;
	\item [2)] $\mathring{\gamma}_{2k-1} \neq \emptyset$ for $k \in \{1, \ldots, n\}$ and every point of an arc $\mathring{\gamma}_{2k-1}$ is a regular boundary point of $f$ (specifically, the restriction of $f$ onto $\gamma_{2k-1}$ is strictly monotone);
	\item [3)] arcs $\gamma_{2k}$, $k \in \{1, \ldots, n\}$ are connected components of level curves of the function $f$.
\end{itemize}
We call such functions \emph{weakly regular on $D$}.
\end{ozn}

\begin{prop}\label{prop_w_r_correctness}
Let $f$ is a weakly regular function on $D$.

A set $\bigcup_{k=1}^n \gamma_{2k}$ does not contain regular boundary points of $f$, therefore the number $\mathcal{N}(f)$ is well defined and coincides with the number of connected components of the set of regular boundary points of $f$.
\end{prop}

\begin{proof}
Let $z \in \Fr{D}$ is a regular boundary point of $f$. Denote by $\Gamma_z$ a connected component of level curve of $f$ which contains $z$. We fix a canonical neighbourhood $U$ of $z$ and a homeomorphism $\psi : U \rightarrow D_{+}^{2}$ from definition~\ref{ozn_f_2}. Then, as it could be easily verified, $\psi^{-1}(\{0\}\times [0,1)) \subseteq \Gamma_z$ and $\emptyset \neq \psi^{-1}(\{0\}\times (0,1)) \subseteq \Gamma_z \cap \Int{D}$. Therefore it follows from the condition 3 of definition~\ref{ozn_weak_regular} that $z \notin \bigcup_{k=1}^n \gamma_{2k}$.

Hence, the number of arcs $\gamma_{2k-1}$, $k \in \{1, \ldots, n\}$ coincides with the number of connected components of the set of regular boundary points of $f$. It depends only on $f$ and the number $\mathcal{N}(f)$ is well defined.
\end{proof}

\begin{lem}\label{lem_prop_of_weak_reg}
Let a function $f$ is weakly regular on $D$.

Every connected component of nonempty level set of $f$ is either a point $z_{2k}$, $k \in \{1, \ldots, n\}$ if $z_{2k} = z_{2k+1}$, or a support of a simple continuous curve $\gamma : I \rightarrow D$ which satisfies to the following properties:
\begin{itemize}
	\item endpoints $\gamma(0)$ and $\gamma(1)$ belong to distinct arcs $\gamma_{2j-1}$ and $\gamma_{2k-1}$, $j, k \in \{1, \ldots, n\}$, $j \neq k$;
	\item either $\gamma(I) \setminus \{\gamma(0), \gamma(1)\} \subset \Int{D}$ or $\gamma(I) = \gamma_{2k}$ for a certain $k \in \{1, \ldots, n\}$.
\end{itemize}
\end{lem}

\begin{proof}
Assume that $c \in \rr$ complies with the inequality $f^{-1}(c) \neq \emptyset$. Let us consider a connected component $\Gamma_c$ of the level set $f^{-1}(c)$. There are two possibilities.

\medskip
{\bfseries 1)} Let $\Gamma_c \cap \Int{D} = \emptyset$. Then $\Gamma_c = \gamma_{2k}$ for a certain $k \in \{1, \ldots, n\}$.

Really, if $\Gamma_c \not\subset \bigcup_{k=1}^n \gamma_{2k}$ then there exists a regular boundary point $w \in \Gamma_c$. It follows from definition~\ref{ozn_f_2} that a portion of the connected component $\Gamma_c$ which is contained in a canonical neighbourhood of the point $w$ has a nonempty intersection with $\Int{D}$.

But if $\Gamma_c \subset \bigcup_{k=1}^n \gamma_{2k}$ then the statement of lemma follows from property 3 of the definition of a weakly regular function on $D$.

In the case under consideration the set $\Gamma_c = \gamma_{2k}$ is either a single-point or a support of a simple continuous curve which endpoints are contained in the sets $\gamma_{2k-1}$ and $\gamma_{2k+1}$ when $k \in \{1, \ldots, n-1\}$ or in $\gamma_{2n-1}$ and $\gamma_1$ if $k=n$.

\medskip
{\bfseries 2)}
Let $\Gamma_c \cap \Int{D} \neq \emptyset$. Then the set $\Gamma_c$ is a support of a simple continuous curve $\gamma : I \rightarrow D$, with $\Gamma_c \cap \Fr{D} = \{\gamma(0), \gamma(1)\} \subset \bigcup_{k=1}^n \mathring{\gamma}_{2k-1}$.

Let us verify this.

It follows from the condition 3 of definition~\ref{ozn_weak_regular} that $\Gamma_c \cap \Fr{D} \subset \bigcup_{k=1}^n \mathring{\gamma}_{2k-1}$. Therefore by definition all points of $\Gamma_c \cap \Fr{D}$ are regular boundary points of $f$.
All remaining points of the set $\Gamma_c$ belong to $\Int D$ and are regular points of $f$.

Denote by $\Theta : (-1, 1) \rightarrow \Int{D^{2}}$ a mapping
\[
\Theta(s) = (0, s), \quad s \in (-1, 1) \,.
\]
It is clear that $\Theta$ is the homeomorphism onto its image.
Denote also
\[
\hat{\Theta} = \restrict{\Theta}{[0, 1)} : [0, 1) \rightarrow D^{2}_{+} \,.
\]
This mapping is obviously also the embedding.

Let $v \in \Int{D} \cap \Gamma_c$. By definition $v$ is the regular point of $f$. Let $U_{v}$ and $\varphi_{v} : U_{v} \rightarrow \Int{D^{2}}$ are a neighbourhood and a homeomorphism from definition~\ref{ozn_f_1}. Then $\varphi_{v}(f^{-1}(f(v))) = \{0\} \times (-1, 1)$, therefore $\varphi_v(\Gamma_c) = \{0\} \times (-1, 1)$, a mapping $\Theta^{-1} \circ \varphi_{v} = \Phi_{v} : Q_{v} = \Gamma_c \cap U_{v} \rightarrow (-1, 1)$ is well defined and it maps $Q_{v}$ homeomorphically onto $(-1, 1)$.
By construction the set $Q_{v}$ is an open neighbourhood of $v$ in the space $\Gamma_c$.

So, a map $(Q_{v}, \Phi_{v} : Q_{v} \rightarrow (-1, 1))$ is associated to every point $v \in \Int{D} \cap \Gamma_c$.

By analogy, if $w \in \Gamma_c \cap \Fr{D}$ then for its neighbourhood $U_{w}$ and a homeomorphism $\psi_{w} : U_{w} \rightarrow D^{2}_{+}$, which comply with definition~\ref{ozn_f_2}, a set $\hat{Q}_{w} = U_{w} \cap \Gamma_c$ and a mapping $\Psi_{w} = \hat{\Theta}^{-1} \circ \psi_{w} : \hat{Q}_{w} \rightarrow [0, 1)$ define a map of the space $\Gamma_c$ in the point $w$.

Obviously the set $\Gamma_c$ with the topology induced from $D$ is a Hausdorff space with a countable base. Moreover, every point of this set has a neighbourhood in $\Gamma_c$ which is homeomorphic to the interval $(0, 1)$ or to the halfinterval $[0, 1)$. Hence $\Gamma_c$ is the compact (it is the closed subset of compact $D$) connected one-dimensional manifold with or without boundary. Therefore the space $\Gamma_c$ is homeomorphic either to the circle $S^{1}$ or to the segment $I$.

Assume that $\Gamma_c \cong S^1$. Let $R \subset D$ is a closed domain with the boundary $\Gamma_c$. All points of $\Int{R}$ are regular points of $f$. From definition~\ref{ozn_f_1} it follows that a regular point cannot be a point of local extremum of $f$. Thus $f \not\equiv const$ on $R$, otherwise every point from $\Int{R}$ should be a point of local extremum of $f$.

$R$ is the compact set, so the continuous function $f$ should rich its maximal and minimal values on $R$. Let $f(v') = \min_{z \in R} f(z)$, $f(v'') = \max_{z \in R} f(z)$ for certain $v'$, $v'' \in R$. We have allready proved that $f(v') \neq f(v'')$, therefore one of these two numbers is distinct from $c = f(\Gamma_c)$ and one of the points $v'$, $v''$ is contained in $\Int{R}$, hence it is the point of local extremum of $f$. Then it cannot be a regular point of $f$.

From the received contradiction we conclude that $\Gamma_c \cong I$, with a pair of points $\{z_{0}(c), z_{1}(c)\} \in \bigcup_{k=1}^n \mathring{\gamma}_{2k-1}$ corresponding to the boundary of the segment and the rest points of $\Gamma_c$ are contained in $\Int{D}$. By definition the function $f$ is strictly monotone on each arc $\mathring{\gamma}_{2k-1}$, $k \in \{1, \ldots, n\}$, therefore $z_0(c) \in \mathring{\gamma}_{2i-1}$, $z_1(c) \in \mathring{\gamma}_{2j-1}$, $i, j \in \{1, \ldots, n\}$ and $i \neq j$.
\end{proof}

\begin{remk}\label{remk_finite_num_of_components}
From condition 2 of Definition~\ref{ozn_weak_regular} and from Lemma~\ref{lem_prop_of_weak_reg} it is seen that every level set of a weakly regular function $f$ has a finite number of connected components in $D$.
\end{remk}

\begin{lem}\label{lem_number_of_arks}
Let $f$ is a weakly regular function on $D$. Then $\mathcal{N}(f) = 2$.
\end{lem}

In order to prove this Lemma we need one simple proposition.

\begin{prop}\label{prop_func_on_compact}
Let $g : K \rightarrow \rr$ is a continuous function on a compact $K$. Then for every $c \in g(K)$ and for a basis $\{U_i\}$ of neighbourhood of $c$ a family of sets $\{W_i = g^{-1}(U_i)\}$ forms the base of neighbourhoods of the level set $g^{-1}(c)$.
\end{prop}

\begin{proof}[Proof of Proposition~\ref{prop_func_on_compact}]
Evidently, it is sufficient to prove that there exists at least one base of neighbourhoods of $c \in g(K)$ full preimages of elements from which form a base of neighbourhoods of the level set $g^{-1}(c)$.

The space $\rr$ complies with the first axiom of countability, so we can assume that the family $\{U_i\}$ is countable.

There exists a countable base $\{\hat{U}_i\}_{i \in \nn}$ of neighbourhoods of $c$ such that
\begin{equation}\label{eq_subsets}
\hat{U}_{i+1} \subseteq \hat{U}_i \,, \quad i \in \nn \,.
\end{equation}
Really, direct verification shows that the family of sets $\hat{U}_{i} = \bigcap_{m=1}^i U_m$, $m \in \nn$ satisfies to our condition.

Suppose that a sequence of sets $\{\hat{W}_i = g^{-1}(\hat{U}_i)\}$ does not form a base of neighbourhoods of $g^{-1}(c)$. Then there exists such a neighbourhood $W$ of this set that the inequality $\hat{W}_i \setminus W \neq \emptyset$ is fulfilled for every $i \in \nn$. We fix $x_i \in \hat{W}_i \setminus W$, $i \in \nn$. From the compactness of $K$ it follows that the sequence $\{x_i\}_{i \in \nn}$ has a convergent subsequence $\{x_{i_j}\}_{j \in \nn}$. Suppose that $x$ is its limit. Relation~\eqref{eq_subsets} assures us that the family of sets $\{\hat{U}_{i_j}\}_{j \in \nn}$ forms the base of neighbourhoods of $c$. Therefore, without loss of generality we can assume that $x = \lim_{i \rightarrow \infty} x_i$.

On one hand the family $\{\hat{U}_i\}$ is the base of neighbourhoods of $c$ and $g(x_n) \in \hat{U}_n$ for every $n \in \nn$. Then it follows from the relation~\eqref{eq_subsets} that also $g(x_k) \in \hat{U}_n$ for every $k > n$, $n \in \nn$. From this and from the continuity of $g$ we get the following equalities $g(x) = \lim_{i \rightarrow \infty} g(x_i) = c$. Hence  $x \in g^{-1}(c) \subset W$.

On the other hand $x \in \Cl{\{x_i \,|\, i \in \nn\}}$ and $\{x_i \,|\, i \in \nn\} \cap W = \emptyset$ by the construction. Therefore the inclusion $x \notin W$ have to be fulfilled.

The contradiction obtained proves proposition.
\end{proof}

\begin{proof}[Proof of lemma~\ref{lem_number_of_arks}]
Let us define for the arc $\gamma_1$ a mapping $\tau : \gamma_1 \rightarrow \Fr{D}$ in the following way. Let $z \in \mathring{\gamma}_1$ and $\Gamma_z \subseteq f^{-1}(f(z))$ is a connected component of a level set of $f$ which contains $z$. We know (see.~Lemma~\ref{lem_prop_of_weak_reg}) that $\Gamma_z$ is a support of a simple continuous curve $\gamma_z : I \rightarrow D$ and that $z$ is one of the endpoints of that curve. Let for example $z = \gamma_z(0)$. We associate to $z$ another endpoint of the curve $\gamma_z$:
\[
\tau(z) = \gamma_z(1) \,, \quad z \in \mathring{\gamma}_1 \,.
\]
Furthermore we set $\tau(z_1) = z_{2n}$, $\tau(z_2) = z_3$.

Let us check that the mapping $\tau$ is continuous on $\gamma_1$.

Suppose first that $z \in \mathring{\gamma}_1$. We designate $c = f(z)$. We know that the level set $f^{-1}(c)$ of $f$ has a finite number of connected components (see Remark~\ref{remk_finite_num_of_components}). Let this number is equal to $l \in \nn$. We fix disjoint open neighbourhoods $W_1, \ldots, W_l$ of these components. Suppose $\Gamma_z \subset W_1$.

It follows from Lemma~\ref{lem_prop_of_weak_reg} and from the condition 3 of Definition~\ref{ozn_weak_regular} that $\tau(z) \in \mathring{\gamma}_{2k-1}$ for some $k \in \{2, \ldots, n\}$. Let $V'$ is an open neighbourhood of $\tau(z)$ in $D$. Without loss of generality we can regard that $V' \cap \Fr{D} \subseteq \gamma_{2k-1} \cup \gamma_1$. Let us also take an open neighbourhood $V$ of $z$ in $D$ such that $V \cap \Fr{D} \subseteq \gamma_1 \cup \gamma_{2k-1}$ and $V \cap \gamma_{2k-1} \subset V'$.

We fix an open neighbourhood $\hat{W}$ of the set $\Gamma_z$ such that $\hat{W} \cap \Fr{D} \subseteq V \cup V'$. For example we can take $\hat{W} = V \cup V' \cup \Int{D}$, where $\Int{D} = D \setminus \Fr{D}$. Designate $W = \hat{W} \cap W_1$. Evidently, inclusions $W \cap \Fr{D} \subseteq V \cup V'$ are valid, moreover $W \cap \gamma_{2k-1} \subset V'$ by the construction.

It follows from Proposition~\ref{prop_func_on_compact} that there exists such $\delta > 0$ for the open neighbourhood $O = W \cup \bigcup_{i=2}^l W_i$ of the level set $f^{-1}(c)$ of $f$ that $Q = f^{-1}(B_{\delta}(c)) \subseteq O$. Here we designate $B_{\delta}(c) = \{t \in \rr \,|\, |t-c| < \delta\}$.

Denote $Q_0 = Q \cap W$, $V_0 = V \cap Q_0$. It is evident that $z \in V_0$ and $Q_0 \cap \Fr{D} \subseteq V_0 \cup V'$.

Let $z' \in \gamma_1 \cap V_0$. Sign by $\Gamma_{z'}$ a connected component of a level set of $f$ which contains $z'$. Let $\gamma_{z'} : I \rightarrow D$ is a simple continuous curve with the support $\Gamma_{z'}$ such that $\gamma_{z'}(0) = z'$ and $\gamma_{z'}(1) = \tau(z')$. Observe that $\Gamma_{z'} \subset Q \subseteq O$, moreover $\Gamma_{z'} \cap Q_0 \neq \emptyset$ and the set $\Gamma_{z'}$ is connected. Open sets $Q_0$ and $Q \cap \bigcup_{i=2}^l W_i$ are disjoint by the construction, so $\Gamma_{z'} \cap \bigcup_{i=2}^l W_i = \emptyset$ and $\Gamma_{z'} \subset Q_0$.

It is easy to see that $\{\gamma_{z'}(0), \gamma_{z'}(1)\} \subset (V_0 \cup V') \cap \Fr{D} \subseteq \gamma_1 \cup \gamma_{2k-1}$, and also $\gamma_{z'}(0) \in \gamma_1$. It is evident that $f \circ \gamma_{z'}(0) = f \circ \gamma_{z'}(1) = f(z')$, hence $\gamma_{z'}(1) \in \gamma_{2k-1}$ (see. Lemma~\ref{lem_prop_of_weak_reg}). But $Q_0 \cap \gamma_{2k-1} \subset W \cap \gamma_{2k-1} \subset V'$, therefore $\tau(z') = \gamma_{z'}(1) \in V'$ and $V_0 \cap \gamma_1 \subseteq \tau^{-1}(V')$.

From arbitrariness in the choice of $z \in \mathring{\gamma}_1$ and of its neighbourhood $V'$ it follows that the mapping $\tau$ is continuous on the set $\mathring{\gamma}_1$.

Suppose now that $z = z_1$ or $z_2$. In the case when $z_1 = z_{2n}$ (respectively $z_2 = z_3$) our previous argument remain true without any changes.

If the arc $\gamma_{2n}$ (respectively $\gamma_2$) does not reduce to a single point then the continuity of $\tau$ in the point $z$ is checked with the help of argument that are analogous to what was stated above. The only essential change is that open sets $V'$ and $V$ should be selected to satisfy correlations $(V' \cup V) \cap \Fr{D} \subseteq \gamma_1 \cup \gamma_{2k-1} \cup \gamma_{2n}$ and $V \cap \gamma_{2k-1} \subset V'$ (respectively $(V' \cup V) \cap \Fr{D} \subseteq \gamma_1 \cup \gamma_{2k-1} \cup \gamma_2$ and $V \cap \gamma_{2k-1} \subset V'$). Also a neighbourhood of the set $\Gamma_z = \gamma_{2n}$ (respectively of $\Gamma_z = \gamma_2$) should be chosen to comply with the inclusion $\hat{W} \cap \Fr{D} \subseteq V \cup V' \cup \Gamma_z$. For example $\hat{W} = V \cup V' \cup \Int{D} \cup \Gamma_z$ will fit.


So, the mapping $\tau : \gamma_1 \rightarrow \Fr{D}$ is continuous. Let us explore some its properties.

The set $\tau(\gamma_1)$ is connected (it is an image of the connected set under continuous mapping) and contains points $z_{2n}$ and $z_3$. Therefore, it should contain one of the arcs of the circle $\Fr{D}$ which connect these points.

Each point of the set $\tau(\gamma_1)$ except $z_{2n}$ and $z_3$ belongs to $\bigcup_{k=2}^n \mathring{\gamma}_{2k-1}$. Really, as we have observed above if $z \in \mathring{\gamma}_1$, then $\tau(z) \in \mathring{\gamma}_{2k-1}$ for a certain $k \neq 1$ (see. Lemma~\ref{lem_prop_of_weak_reg}).

By definition $\mathring{\gamma}_i \cap \gamma_j = \emptyset$ when $i \neq j$, therefore
\begin{equation}\label{eq_cond_1}
\tau(\gamma_1) \cap \mathring{\gamma}_1 = \emptyset \,.
\end{equation}

If $n \geq 3$, then
\begin{equation}\label{eq_cond_2}
\gamma_4 \cap \tau(\gamma_1) = \emptyset \,.
\end{equation}
This is the consequence of a simple observation that $\{z_3, z_{2n}\} \cap \gamma_4 = \emptyset$ when $n \geq 3$ (see. condition 2 of Definition~\ref{ozn_weak_regular}) together with the relation $\tau(\gamma_1) \subseteq \{z_3, z_{2n}\} \cup \bigcup_{k=2}^n \mathring{\gamma}_{2k-1}$.

To complete the proof of lemma it remains to notice that if $n \geq 3$ then the nonempty sets $\mathring{\gamma}_1$ and $\gamma_4$ are contained in different connected components of $\Fr{D} \setminus \{z_3, z_{2n}\}$ and relations~\eqref{eq_cond_1} and~\eqref{eq_cond_2} could not hold at the same time, otherwise points $z_3$ and $z_{2n}$ would belong to different connected components of the set $\tau(\gamma_1)$.

Thus, $n = \mathcal{N}(f) = 2$.
\end{proof}

\begin{ozn}\label{ozn_almost_weak_regular}
Let for some $n \geq 2$ and for a sequence of points $z_1, \ldots, z_{2n} \in \Fr{D}$ a function $f$ complies with all conditions of Definition~\ref{ozn_weak_regular} except condition 3, instead of which the following condition is valid
\begin{itemize}
	\item[$3')$] for $j = 2k$, $k \in \{1, \ldots, n\}$ the arc $\gamma_j$ belongs to a level set of $f$.
\end{itemize}
We shall call such a function \emph{almost weakly regular on $D$}.
\end{ozn}

Let $f$ is a weakly regular function on $D$. We denote by $2\cdot\mathcal{N}(f)$ the minimal number of points and arcs which satisfy to Definition~\ref{ozn_almost_weak_regular}. Obviously, this number is well defined and depends only on $f$.

\begin{prop}\label{prop_min_num_of_arcs}
Suppose that for a certain $n \geq 2$ and a sequence of points $z_1, \ldots, z_{2n} \in \Fr{D}$ function $f$ complies with conditions of Definition~\ref{ozn_almost_weak_regular}. If $n = \mathcal{N}(f)$, then a family of sets $\{\mathring{\gamma}_{2k-1}\}_{k=1}^n$ coincides with the family of connected components of the set of regular boundary points of $f$.
\end{prop}

\begin{proof}
Let us designate a set of regular boundary points of $f$ by $R$. The set $R$ is open in the space  $\Fr{D}$ by definition, therefore its connected components are open arcs of the circle $\Fr{D}$.

Let us check that $R \cap \bigcup_{k=1}^n \mathring{\gamma}_{2k} = \emptyset$.

Really, for an arbitrary point $z \in \mathring{\gamma}_{2k}$ there exists its open neighbourhood small enough to comply with the inequality $U(z) \cap \Fr{D} \subseteq \mathring{\gamma}_{2k}$, hence from the condition $3'$ of Definition~\ref{ozn_almost_weak_regular} it follows that $U(z) \cap \Fr{D} \subseteq f^{-1}(z)$ and a canonical neighbourhood $V(z) \subseteq U(z)$ of $z$ in the sense of Definition~\ref{ozn_f_2} can not exist (see also Remark~\ref{remk_kanon_neighb}).

Let us verify that if $\gamma_{2k} \cap R \neq \emptyset$ for some $k \in \{1, \ldots, n\}$ then $\mathring{\gamma}_{2k} = \emptyset$ and $\gamma_{2k} = \{z_{2k}\}$.

Let $\gamma_{2k} \cap R \neq \emptyset$. Then $\gamma_{2k} \cap R \subseteq \{z_{2k}, z_{m}\} = \gamma_{2k} \setminus \mathring{\gamma}_{2k}$, where $m \equiv 2k+1 \pmod{2n}$. But it is easy to see that if $\mathring{\gamma}_{2k} \neq \emptyset$ then for an arbitrary neighbourhood $U$ of $z_{2k}$ in the space $D$ an intersection $U \cap \mathring{\gamma}_{2k}$ is not empty and contains some point $z' \neq z_{2k}$. Therefore $\{z_{2k}, z'\} \subset f^{-1}(f(z)) \cap U \cap \Fr{D}$ and $U$ can not be a canonical neighbourhood of $z_{2k}$ in the sense of Definition~\ref{ozn_f_2}. Similar is also true for $z_m$. Consequently, if $\mathring{\gamma}_{2k} \neq \emptyset$ then $\{z_{2k}, z_{m}\} \cap R = \emptyset$ and $\gamma_{2k} \cap R = \emptyset$.

Let $n = \mathcal{N}(f)$.

Let us check that $R \cap \bigcup_{k=1}^n \gamma_{2k} = \emptyset$.

Really, if $\gamma_{2k} \cap R \neq \emptyset$ for some $k \in \{1, \ldots, n\}$ then $z_{2k} = z_m$, $m \equiv 2k+1 \pmod{2n}$ and $\gamma_{2k} = \{z_{2k}\} \subset R$. Then the open arc $\mathring{\gamma}_{2k-1} \cup \gamma_{2k} \cup \mathring{\gamma}_m$ is contained in $R$ so we can throw off the points $z_{2k}$, $z_m$ and replace three consequent arcs $\gamma_{2k-1}$, $\gamma_{2k}$, $\gamma_m$, $m \equiv 2k+1 \pmod{2n}$ by the arc $\gamma_{2k-1} \cup \gamma_{2k} \cup \gamma_m$ in order to reduce the quantity of points and corresponding arcs in the collection $\{z_1, \ldots, z_{2n}\}$. But it is impossible since the quantity of points $2n$ is already minimal.

It is obvious that $\Fr{D} = \bigcup_{k=1}^n \gamma_{2k} \cup \bigcup_{k=1}^n \mathring{\gamma}_{2k+1}$ and $\bigcup_{k=1}^n \mathring{\gamma}_{2k+1} \subseteq R$, therefore
\[
R = \bigcup_{k=1}^n \mathring{\gamma}_{2k+1}
\]
and the family $\{\mathring{\gamma}_{2k-1}\}_{k=1}^n$ of disjoint nonempty connected sets which are open in $\Fr{D}$ coincides with the family of connected components of the set $R$ of regular boundary points of $f$.
\end{proof}

\begin{lem}\label{lem_weak_regular_2}
If $f : D \rightarrow \rr$ is almost weakly regular on $D$ and $\mathcal{N}(f) = 2$, then $f$ is weakly regular on $D$.
\end{lem}

\begin{proof}
If $\mathcal{N}(f) = 2$, then the frontier $\Fr{D}$ of $D$ consists of four arcs $\gamma_1, \ldots, \gamma_4$, where arcs $\gamma_1$ and $\gamma_3$ are nondegenerate and $f$ is strictly monotone on them. On each of the arcs $\gamma_2$ and $\gamma_4$ function $f$ is constant and each of these arcs can degenerate into a point. Let $\gamma_2 \subseteq f^{-1}(c')$, $\gamma_4 \subseteq f^{-1}(c'')$.
From the strict monotony of $f$ on $\gamma_1$ we conclude that $c'' = f(z_1) = f(\gamma_4) \neq f(\gamma_2) = f(z_2) = c'$. Let $c' < c''$ for definiteness.

Every interior point of $D$ is regular, hence local extremum points of $f$ can be situated only on the frontier $\Fr{D}$. From what we said above it follows that $f(D) = [c', c'']$ and every point of the set $f^{-1}(c') \cup f^{-1}(c'')$ is a local extremum point of $f$ on $D$. Therefore $f^{-1}(c') \cup f^{-1}(c'') \subset \Fr{D}$. But $f^{-1}(c') \cap \Fr{D} = \gamma_2$ and $f^{-1}(c'') \cap \Fr{D} = \gamma_4$. Consequently $f^{-1}(c') = \gamma_2$, $f^{-1}(c'') = \gamma_4$ and $f$ is weakly regular on $D$.
\end{proof}

\begin{figure}
\centerline{\includegraphics{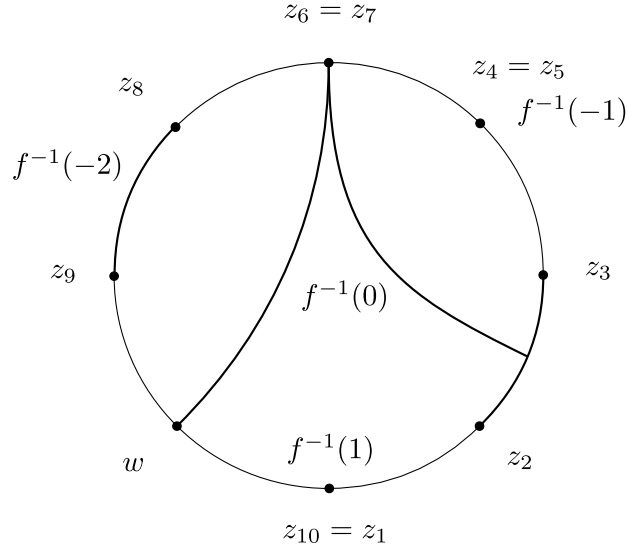}}
\caption{An almost weakly regular on $D$ function $f$ with $\mathcal{N}(f) = 5$. $w \in \gamma_9$ is a regular boundary point of $f$.}\label{Fig_N_gt_2}
\end{figure}

\begin{remk}
There exist almost weakly regular on $D$ functions with $\mathcal{N}(f) > 2$, see Figure~\ref{Fig_N_gt_2}.
\end{remk}

%% file: level_sets.tex
\section{On level sets of weakly regular functions on the square $I^2$.}

Let $W$ be a domain in the plane $\rr^{2}$, $f : \Cl{W} \rightarrow \rr$ be a continuous function.
\begin{ozn}\label{ozn_f_3}
A simple continuous curve $\gamma : [0, 1] \rightarrow \Cl{W}$ is called an \emph{$U$-trajectory} if $f \circ \gamma$ is strongly monotone on the segment $[0, 1]$.
\end{ozn}

We designate $I = [0, 1]$, $I^{2} = I \times I \subseteq \rr^{2}$, $\mathring{I}^{2} = \Int{I^{2}} = (0, 1) \times (0, 1)$.

Let us consider a continuous function $f : I^{2} \rightarrow \rr$ which complies with the following properties:
\begin{itemize}
	\item $f([0, 1] \times \{0\}) = 0$, $f([0, 1] \times \{1\}) = 1$;
	\item each point of the set $\mathring{I}^{2}$ is a regular point of $f$;
	\item every point of $\{0, 1\} \times (0, 1)$ is a regular boundary point of $f$;
	\item for any point of a dense subset $\Gamma$ of $(0, 1) \times \{0, 1\}$ there exists an $U$-trajectory which goes through this point.
\end{itemize}

\begin{prop}\label{prop_f_0}
Function $f$ is weakly regular on the square $I^2$.
\end{prop}

\begin{proof}
We take $z_1 = (1, 0)$, $z_2 = (1, 1)$, $z_3 = (0, 1)$, $z_4 = (0, 0)$. It is obvious that $f$ is almost weakly regular on $I^2$ for this sequence of points and that $\mathcal{N}(f) = 2$. Then as a consequence of Lemma~\ref{lem_weak_regular_2} this function is weakly regular on the square $I^2$.
\end{proof}

\begin{nas}\label{nas_prop_f_0}
$f(z) \in (0, 1)$ for all $z \in I \times (0, 1)$. Moreover
\begin{itemize}
	\item for every $c \in (0, 1)$ level set $f^{-1}(c)$ is a support of a simple continuous curve $\zeta_{c} : I \rightarrow I^{2}$ such that $\zeta_{c}(0) \in \{0\} \times (0, 1)$, $\zeta_{c}(1) \in \{1\} \times (0, 1)$, $\zeta_{c}(t) \in \mathring{I}^{2}$ $\forall t \in (0, 1)$;
	\item level sets $f^{-1}(0) = I \times \{0\}$ and $f^{-1}(1) = I \times \{1\}$ are supports of simple continuous curves.
\end{itemize}
\end{nas}

\begin{proof}
This statement follows from Lemma~\ref{lem_prop_of_weak_reg}.
\end{proof}

\begin{lem}\label{lemma_f_2}
Let $v \in I^{2}$. For every $\varepsilon > 0$ there exists $\delta > 0$ to satisfy the following property:
\begin{itemize}
	\item[(ELC)] if a set $f^{-1}(c)$ is support of a simple continuous curve $\zeta_{c} : I \rightarrow I^{2}$ for a certain $c \in (0, 1)$ and $\zeta_{c}(s_{1})$, $\zeta_{c}(s_{2}) \in U_{\delta}(v) = \{ z \;|\; |z-v| < \delta \}$ for some $s_{1}$, $s_{2} \in I$, $s_{1} < s_{2}$, then $\zeta_{c}(t) \in U_{\varepsilon}(v)$ for all $t \in [s_{1}, s_{2}]$.
\end{itemize}
\end{lem}

\begin{remk}
Fulfillment of the {\upshape (ELC)} condition is an analog of so called \emph{equi-locally-connectedness} of a family of level sets of $f$ in a point $v \in I^{2}$ (see~\cite{Toki_51}).
\end{remk}

\begin{proof}
Let contrary to Lemma statement there exist $\varepsilon > 0$, a sequence $\{d_{j}\}$ of function $f$ values, a family $\{\zeta_{j}\}$ of simple Jordan curves with supports $\{f^{-1}(d_{j})\}$, and also sequences $\{s_{j}'\}$, $\{s_{j}''\}$ and $\{\tau_{j}\}$ of parameter values, such that correlations hold true
\begin{gather*}
s_{j}' < \tau_{j} < s_{j}'' \quad \forall j \in \nn \,, \\
\lim_{j \rightarrow \infty} \zeta_{j}(s_{j}') = \lim_{j \rightarrow \infty} \zeta_{j}(s_{j}'') = v \,, \\
dist (\zeta_{j}(\tau_{j}), v) \geq \varepsilon \quad \forall j \in \nn \,.
\end{gather*}
We shall denote $v_{j}' = \zeta_{j}(s_{j}')$, $v_{j}'' = \zeta_{j}(s_{j}'')$, $w_{j} = \zeta_{j}(\tau_{j})$, $j \in \nn$.

From the compactness of square it follows that the sequence $\{w_{j}\}$ has at least one limit point. Going over to a subsequence we can assume that this sequence is convergent. Let its limit is $w$. The continuity of $f$ implies
\[
d = \lim_{i \rightarrow \infty} d_{j} = \lim_{i \rightarrow \infty} f(\zeta_{j}(\tau_{j})) = f(w) = f(v) \,.
\]

Let us fix a simple continuous curve $\zeta_{d} : I \rightarrow I^{2}$ with the support $f^{-1}(d)$. Then $v = \zeta_{d}(s)$, $w = \zeta_{d}(\tau)$ for certain values of parameter $s$, $\tau \in I$, $s \neq \tau$.

We consider the following possibilities.

\medskip
\emph{Case 1.} Let $d \notin \{0, 1\}$.

We fix $t_{0} \in I$ such that one of pairs of inequalities $s < t_{0} < \tau$ or $\tau < t_{0} < s$ holds true. Designate $z_{0} = \zeta_{d}(t_{0})$. We note that it follows that $t_{0} \notin \{0, 1\}$ from the choice of $t_{0}$, therefore Corollary~\ref{nas_prop_f_0} implies inequality $z_{0} \notin \{0, 1\} \times I$, which in turn has as a consequence inclusion $z_{0} \in \mathring{I}^{2} = \Int{I^{2}}$.

Definition~\ref{ozn_f_1} implies that for a certain $\alpha > 0$ through $z_{0}$ passes an $U$-trajectory $\gamma_{0} : I \rightarrow I^{2}$ such that $\gamma_{0}(0) \in f^{-1}(d-\alpha)$,  $\gamma_{0}(1) \in f^{-1}(d+\alpha)$, $\gamma_{0}(1/2) = z_{0}$. Moreover, if necessary we can decrease $\alpha$ as much that the curve $\gamma_{0}$ will not intersect lateral sides of the square $I^{2}$.

Let us consider a curvilinear quadrangle $J$ bounded by Jordan curves $\zeta_{d-\alpha} = f^{-1}(d-\alpha)$, $\zeta_{d+\alpha} = f^{-1}(d+\alpha)$, $\eta_{0} = f^{-1}([d-\alpha, d+\alpha]) \cap (\{0\} \times I)$, $\eta_{1} = f^{-1}([d-\alpha, d+\alpha]) \cap (\{1\} \times I)$. It is clear that this quadrangle is homeomorphic to closed disk.

Ends $\zeta_{d}(0)$ and $\zeta_{d}(1)$ of the Jordan curve $\zeta_{d}$ are contained in lateral sides of $J$, namely $\zeta_{d}(0) \in \eta_{0}$, $\zeta_{d}(1) \in \eta_{1}$ (see Corollary~\ref{nas_prop_f_0}). From the other side by construction the curve $\gamma_{0}$ is a cut of the quadrangle $J$ between the points $\gamma_{0}(0) \in \zeta_{d-\alpha}$ and $\gamma_{0}(1) \in \zeta_{d+\alpha}$ which are contained in its bottom and top side respectively.

From what we said above it follows that the set $J \setminus \gamma_{0}(I)$ has two connected components $J_{0}$ and $J_{1}$, moreover $\eta_{0}$ and $\eta_{1}$ are contained in different components. Let $\eta_{0} \subseteq J_{0}$, $\eta_{1} \subseteq J_{1}$.

It is obvious that $\gamma_{0}(I) \cap \zeta_{d}(I) = \{z_{0}\} = \{\zeta_{d}(t_{0})\}$. Hence points $v$ and $w$ belong to different components of $J \setminus \gamma_{0}(I)$.

Really, if $s < t_{0} < \tau$ then $\zeta_{d}([0, s]) \subseteq J_{0}$, $\zeta_{d}([\tau, 1]) \subseteq J_{1}$, because $\zeta_{d}([0, s])$, $\zeta_{d}([\tau, 1]) \subseteq J \setminus \gamma_{0}(I)$, these sets are connected and inequalities are fulfilled $\emptyset \neq \zeta_{d}([0, s]) \cap J_{0} \ni \zeta_{d}(0)$, $\emptyset \neq \zeta_{d}([\tau, 1]) \cap J_{1} \ni \zeta_{d}(1)$. By analogy, if $\tau < t_{0} < s$ then $\zeta_{d}([0, \tau]) \subseteq J_{0}$ and $\zeta_{d}([s, 1]) \subseteq J_{1}$.

Let $V$ and $W$ are open neighbourhoods of the points $v$ and $w$ respectively, and one of these sets does not intersect $\Cl{J_{0}}$, the other has an empty intersection with $\Cl{J_{1}}$. Existence of such neighbourhoods is a consequence from the following argument: if for a certain $m \in \{0, 1\}$ the point $z$ does not belong neither to the set $J_{m}$, nor to the curve $\gamma_{0}$, then $z \in \Int{(\rr^{2} \setminus J_{m})}$ since $\Cl{J_{m}} = J_{m} \cup \gamma_{0}(I)$.

So, one of the sets $V_{0} = V \cap J$, $W_{0} = W \cap J$ belongs to $J_{0}$, other is contained in $J_{1}$.

Fix so big $k \in \nn$ that $v_{k}'$, $v_{k}'' \in V$, $w_{k} \in W$, $d_{k} \in (d-\alpha, d+\alpha)$. Then $v_{k}'$, $v_{k}''$, $w_{k} \in \zeta_{k}(I) = f^{-1}(d_{k}) \subseteq J$ and $v_{k}'$, $v_{k}'' \in V_{0}$, $w_{k} \in W_{0}$. Thus the ends of both simple continuous curves $\zeta_{k}([s_{k}', \tau_{k}])$ and $\zeta_{k}([\tau_{k}, s_{k}''])$ are contained in different connected components of $J \setminus \gamma_{0}$. Therefore there exist $t' \in (s_{k}', \tau_{k})$, $t'' \in (\tau_{k}, s_{k}'')$ such that $\zeta_{k}(t')$, $\zeta_{k}(t'') \in \gamma_{0}(I)$.

By construction we have $\zeta_{k}(t') \neq \zeta_{k}(t'')$, but this is impossible since the arc $\gamma_{0}$ is $U$-trajectory and should intersect a level set $f^{-1}(d_{k}) = \zeta_{k}(I)$ not more than in one point.
This brings us to the contradiction with our initial assumptions and proves Lemma in the case~1.

\medskip
\emph{Case 2.} Let $d \in \{0, 1\}$.

Obviously, $\zeta_d(I)$ is a connected component of the set $I \times \{0, 1\}$. Therefore the set $\Gamma \cap \zeta_d(I)$ is dense in $\zeta_d(I)$. Mapping $\zeta_d$ is homeomorphism onto its image, hence the set $\Gamma' = \zeta_d^{-1}(\Gamma \cap \zeta_d(I))$ is dense in segment. We fix $t_{0} \in \Gamma'$ such that one of the following pairs of inequalities $s < t_{0} < \tau$ or $\tau < t_{0} < s$ is fulfilled. Denote $z_{0} = \zeta_{d}(t_{0})$. By the choice of $t_0$ there exists a $U$-trajectory which passes through $z_0$.

Further on this case is considered by analogy with case~1 with evident changes.
\end{proof}


Let us recall one important definition (see~\cite{Whitney_33, Morse_36}). Let $\alpha$, $\beta : I \rightarrow \rr^{2}$ be continuous curves. We designate by $Aut_{+}(I)$ a set of all orientation preserving homeomorphisms of the segment onto itself. For every $H \in Aut_{+}(I)$ ($H(0)=0$) we sign
\[
D(H) = \max_{t \in I} \dist (\alpha(t), \beta \circ H(t)) \,.
\]

\begin{ozn}
Value
\[
\df (\alpha, \beta) = \inf_{H \in Aut_{+}(I)} D(H)
\]
is called a \emph{Frechet distance} between curves $\alpha$ and $\beta$.
\end{ozn}

For every value $c \in I$ of a function $f$ we can fix a parametrization $\zeta_{c} : I \rightarrow \rr^2$ of the level set $f^{-1}(c)$ in such way that an inclusion $\zeta_{c}(0) \in \{0\} \times I$ holds true (see Corollary~\ref{nas_prop_f_0}).
The following statement is valid.
\begin{lem}\label{lemma_f_3}
Let $c \in I$.
For every $\varepsilon > 0$ there exists $\delta > 0$ such that $\df (\zeta_c, \zeta_d) < \varepsilon$ when $|c-d| < \delta$.
\end{lem}

\begin{proof}
Let $c \in I$, $\zeta_{c} : I \rightarrow I^{2}$ is a simple continuous curve with a support $f^{-1}(c)$. Let $\varepsilon > 0$ is given.

Let us find for every $t \in I$ a number $\delta(t) > 0$ which satisfies Lemma~\ref{lemma_f_2} for a point $\zeta_{c}(t)$ and $\hat{\varepsilon} = \varepsilon/2$.

We consider two possibilities.

\medskip
\emph{Case 1}. Let $c \in (0, 1)$. It is clear that for every $t \in I$ there exists a neighbourhood $U(t)$ of $\zeta_{c}(t)$ which complies with the following conditions:
\begin{itemize}
	\item $U(t) \subseteq U_{\delta(t)}(\zeta_{c}(t))$;
	\item $U(t)$ is a canonical neighbourhood from Definition~\ref{ozn_f_1} when $t \in (0, 1)$ or from Definition~\ref{ozn_f_2} for $t \in \{0, 1\}$.
\end{itemize}

Let a family of sets
\[
U_{0} = U(0), \quad U_{1} = U(t_{1}), \ldots, U_{n-1} = U(t_{n-1}), \quad U_{n} = U(1) \,,
\]
forms a finite subcovering of a covering $\{U(t)\}_{t \in I}$ of the compact $f^{-1}(c)$.

We denote $z_{i} = \zeta_{c}(t_{i})$, $J_{i} = \zeta_{c}^{-1}(U_{i} \cap f^{-1}(c))$, $i \in \{0, \ldots, n\}$. By construction a family of sets $\{J_{i}\}_{i=0}^{n}$ is a covering of $I$. From Definitions~\ref{ozn_f_1} and~\ref{ozn_f_2} it follows that
\[
J_{0} \cong [0, 1) \,, \quad
J_{n} \cong (0, 1] \,; \qquad
J_{i} \cong (0, 1) \,, \quad i \in \{1, \ldots, n-1\} \,.
\]

If necessary we decrease neighbourhoods $U_{i}$ as much that on one hand they remain canonical and form a covering of $f^{-1}(c)$ as before, on the other hand no two different intervals from the family $\{J_{i}\}_{i = 0}^{n}$ should have a common endpoint.

It is straightforward that there exists a finite sequence of numbers $0=\tau_{0} < \tau_{1} < \ldots < \tau_{m-1} < \tau_{m}=1$, which satisfies a condition:
\begin{itemize}
	\item for every $k \in \{1, \ldots, m\}$ there exists $i(k) \in \{0, \ldots, n\}$ such that $\tau_{k-1}$, $\tau_{k} \in J_{i(k)}$.
\end{itemize}
We fix such a family $\{\tau_{k}\}_{k=0}^{m}$ and denote by $\theta : \{1, \ldots, m\} \rightarrow \{0, \ldots, n\}$ a mapping $\theta : k \mapsto i(k)$. We also designate $w_{k} = \zeta_{c}(\tau_{k})$, $k \in \{0, \ldots, m\}$.

From Definitions~\ref{ozn_f_1} and~\ref{ozn_f_2} it follows that through every point $w_{k}$, $k \in \{0, \ldots, m\}$ passes an $U$-trajectory $\gamma_{k} : I \rightarrow I^{2}$ which complies with inequalities $f \circ \gamma_{k}(0) < c < f \circ \gamma_{k}(1)$. We can also assume that $\gamma_{0}(I) \subset \{0\} \times I$ and $\gamma_{m}(I) \subset \{1\} \times I$ (see Definition~\ref{ozn_f_2}). If necessary we decrease these $U$-trajectories as much that they should be pairwise disjoint and for every $k \in \{0, \ldots, m\}$ relations $\gamma_{k-1}(I)$, $\gamma_{k}(I) \subset U_{\theta(k)}$ should hold true (that can be done since the curves $\gamma_{k}$ are continuous and by construction inclusions $w_{k-1}$, $w_{k} \in U_{\theta(k)}$ are valid). Let us designate
\[
\delta = \min_{k \in \{0, \ldots, m\}}
\min ( |f \circ \gamma_{k}(0) - c|, |f \circ \gamma_{k}(1) - c| ) \,.
\]

Suppose that an inequality $|c-d| < \delta$ holds true. Then by construction a simple continuous curve $\zeta_{d} : I \rightarrow I^{2}$ with the support $f^{-1}(d)$ must intersect every $U$-trajectory $\gamma_{k}$ in a single point $w_{k}^{d}$. Denote $\tau_{k}^{d} = \zeta_{d}^{-1}(w_{k}^{d})$, $k \in \{0, \ldots, m\}$.

By choice of parameterization of curves $\zeta_{c}$ and $\zeta_{d}$ we have $\zeta_{c}(j)$, $\zeta_{d}(j) \in \{j\} \times I$, $j = 0, 1$. Therefore $w_{k}^{d} \in \gamma_{k}(I)$ when $k = 0$ or $m$, and $\tau_{0}^{d} = 0$, $\tau_{m}^{d} = 1$.

We designate $K = [\min(c, d), \max(c, d)]$. Let us consider a curvilinear quadrangle $R$ bounded by curves $\zeta_{c}$, $\gamma_{0}(I) \cap f^{-1}(K)$, $\zeta_{d}$, $\gamma_{m}(I) \cap f^{-1}(K)$. Curves $\gamma_{k}(I) \cap f^{-1}(K)$, $k \in \{1, \ldots, m-1\}$ form cuts of this quadrangle between top and bottom sides and are pairwise disjoint. The straightforward consequence of this fact is that corresponding endpoints $\{w_{k}\}$ and $\{w_{k}^{d}\}$ of these cuts are similarly ordered on the curves $\zeta_{c}$ and $\zeta_{d}$. Therefore
\[
0 = \tau_{0}^{d} < \tau_{1}^{d} < \ldots < \tau_{m-1}^{d} < \tau_{m}^{d} = 1 \,.
\]

Let a mapping $H : I \rightarrow I$ translates $\tau_{k}$ to $\tau_{k}^{d}$ for every $k$ and a segment $[\tau_{k-1}, \tau_{k}]$ linearly maps onto $[\tau_{k-1}^{d}, \tau_{k}^{d}]$, $k \in \{1, \ldots, m\}$. It is clear that $H \in Aut_{+}(I)$.

Let us estimate the value of $D(H)$. By construction for every $k \in \{1, \ldots, m\}$ we have
\[
w_{k-1}, w_{k}, w_{k-1}^{d}, w_{k}^{d} \in U_{\theta(k)} \,,
\]
therefore, it follows from the choice of neighbourhood $U_{\theta(k)}$ of the point $z_{\theta(k)}$ and from Lemma~\ref{lemma_f_2} that
\[
\zeta_{c}([\tau_{k-1}, \tau_{k}]), \zeta_{d}([\tau_{k-1}^{d}, \tau_{k}^{d}]) \subset U_{\varepsilon/2}(z_{\theta(k)})
\]
and for every $t \in [\tau_{k-1}, \tau_{k}]$ an inequality $\dist (\zeta_{c}(t), \zeta_{d} \circ H(t)) < \varepsilon$ holds true.

From what was said above we make a consequence that
\[
\df (\zeta_{c}, \zeta_{d}) \leq D(H) = \max_{k \in \{1, \ldots, m\}} \max_{t \in [\tau_{k-1}, \tau_{k}]} \dist (\zeta_{c}(t), \zeta_{d} \circ H(t)) < \varepsilon \,,
\]
if $|c-d| < \delta$.

\medskip
\emph{Case 2}. Let $c \in \{0, 1\}$. In this case proof mainly repeats argument of the previous case with the following changes.

We know already that a set $\Gamma' = \zeta_c^{-1}(\Gamma \cap \zeta_c(I))$ is dense in segment (see the proof of Lemma~\ref{lemma_f_2}). Moreover, every point of the set $\{0, 1\} \times (0, 1)$ is a regular boundary point of $f$. Therefore, on each of lateral sides of the square $f$ is strongly monotone, hence both of lateral sides of the square are supports of $U$-trajectories, and $0, 1 \in \Gamma'$.

The set $\zeta_c(I)$ in the case under consideration is the linear segment, so we can select a covering $\{U(t)\}_{t \in I}$ from the following reason:
\begin{itemize}
	\item $U(t) = U_{\delta(t)}(\zeta_c(t))$ for $t = 0, 1$;
	\item $U(t) = U_{\delta'(t)}(\zeta_c(t))$, where $\delta'(t) < \min(\delta, t, 1-t)$ when $t \in (0, 1)$.
\end{itemize}

After the choice of numbers $0 = \tau_0 < \tau_1 < \ldots < \tau_m = 1$ is done, we can with the help of small perturbations of $\tau_1, \ldots, \tau_{m-1}$ achieve that $\{\tau_0, \ldots, \tau_m\} \subset \Gamma'$ and a family $\{\tau_k\}$ keeps its properties (see case~1). Then for every $k \in \{0, \ldots, m\}$ there exists an $U$-trajectory which passes through $\zeta_c(\tau_k)$.

Subsequent proof repeats the argument of case~1.
\end{proof}

Let us remind several important definitions.

Let $\lambda : I \rightarrow \rr^{2}$ is a continuous curve. For every $n \in \nn$ we designate by $S_{n}(\lambda)$ a set of all sequences $(p_{i} \in \lambda(I))_{i=0}^{n}$ of the length $n+1$, such that $p_{i} = \lambda(t_{i})$, $i = 0, \ldots, n$, and inequalities $t_{0} \leq t_{1} \leq \ldots \leq t_{n}$ hold true. Denote
\[
d(p_{0}, \ldots, p_{n}) = \min_{i=1, \ldots ,n} \dist(p_{i-1}, p_{i}) \,.
\]

\begin{ozn}[see~\cite{Whitney_33, Morse_36}]
Let $\lambda : I \rightarrow \rr^{2}$ be a continuous curve,
\[
\mu_{n}(\lambda) = \sup_{(p_{0}, \ldots, p_{n}) \in S_{n}(\lambda)} d(p_{0}, \ldots, p_{n}) \,, \quad n \in \nn \,.
\]
A value
\[
\mu_{\lambda} = \sum_{n \in \nn} \frac{\mu_{n}(\lambda)}{2^{n}}
\]
is called \emph{$\mu$-length} of $\lambda$.
\end{ozn}

Let again $\lambda : I \rightarrow \rr^{2}$ is a continuous curve. We consider a family of continuous curves $\lambda_{t} : I \rightarrow \rr^{2}$, $\lambda_{t}(\tau) = \lambda(t\tau)$, $t \in I$. Let $\mu(t) = \mu_{\lambda_{t}}$, $t \in I$, is a $\mu$-length of the curve $\lambda$ from 0 to $t$. It is known that $\mu$ continuously and monotonically maps $I$ onto $[0, \mu_{\lambda}]$. It is found that for an arbitrary continuous curve $\lambda$ and for every $c \in [0, \mu_{\lambda}]$ a set $\lambda(\mu^{-1}(c))$ is singleton. Hence a mapping $r_{\lambda} : [0, \mu_{\lambda}] \rightarrow \lambda(I) \subset \rr^{2}$, $r_{\lambda}(c) = \lambda(\mu^{-1}(c))$, is well defined. It is known also that this mapping is continuous.

\begin{ozn}[see~\cite{Morse_36}]
A curve $r_{\lambda}$ is called a \emph{$\mu$-parameterization} of $\lambda$.
\end{ozn}

We say that a continuous curve $\eta : I \rightarrow \rr^{2}$ is \emph{derived} from a continuous curve $\lambda : I \rightarrow \rr^{2}$ if there exists such a continuous nondecreasing surjective mapping $u : I \rightarrow I$ that $\eta(t) = \lambda \circ u(t)$, $t \in I$. It is known that an arbitrary curve $\lambda$ is derived from its $\mu$-parameterization $r_{\lambda}$ (see~\cite{Morse_36}). Therefore, if $\lambda$ is a simple continuous curve, then $r_{\lambda}$ is also a simple continuous curve.

\begin{ozn}[see~\cite{Morse_36}]
\emph{Class of curves} is a family of all continuous curves with the same $\mu$-parameterization.
\end{ozn}

It turns out (see~\cite{Morse_36}) that the Frechet distance between curves does not change when we replace curves to other representatives of their class of curves. Consequently Frechet distance is well defined on the set of all classes of curves. Moreover it is known that Frechet distance is the distance function on this set. We shall denote metric space of classes of curves with the Frechet distance by $\cM$.

We consider a set $R \subseteq \cM \times \rr$,
\[
R = \bigcup_{\lambda \in \cM} \{ (\lambda, \tau) \,|\, \tau \in [0, \mu_{\lambda}] \} \,,
\]
and a correspondence $q : R \rightarrow \rr^{2}$,
\[
q(\lambda, \tau) = r_{\lambda}(\tau) \,, \quad (\lambda, \tau) \in R \,,
\]
which maps a pair $(\lambda, \tau)$ to a point of the curve $\lambda$ such that $\mu$-length of $\lambda$ from $\lambda(0)$ to this point equals $\tau$. It is known (see~\cite{Morse_36}) that the mapping $q$ is continuous. This allows us to prove following.

\begin{lem}\label{lemma_f_4}
Let $\varphi : I \rightarrow \cM$ be a continuous mapping such that $\mu_{\varphi(t)} > 0$ for every $t \in I$.


Then a map $\Phi : I^{2} \rightarrow \rr^{2}$,
\[
\Phi (\tau, t) = r_{\varphi(t)}(\mu_{\varphi(t)} \cdot \tau) \,, \quad (\tau, t) \in I^{2} \,,
\]
is continuous and for any $t \in I$ correlation $\Phi(I \times \{t\}) = \varphi(t)(I)$ holds true.
\end{lem}

Before we begin to prove Lemma we will check following statement.

\begin{prop}\label{prop_f_1}
Let $[a, b] \subseteq \rr$ and $\alpha$, $\beta : [a, b] \rightarrow \rr$ be such continuous functions that $\alpha(x) < \beta(x)$ for every $x \in [a, b]$. Let
\[
K = \{ (x, y) \in \rr^{2} \,|\, x \in [a, b], y \in [\alpha(x), \beta(x)] \} \,.
\]

Then a mapping $G : [a, b] \times I \rightarrow K$,
\[
G(x, t) = (x, t\beta(x) + (1-t)\alpha(x))
\]
is homeomorphism.
\end{prop}

\begin{proof}
We shall consider $G$ as a mapping $[a, b] \times I \rightarrow \rr^{2}$.

It is known (see~\cite{RF}) that a mapping $\Phi : X \rightarrow \prod_{\alpha} Y_{\alpha}$ is continuous iff a coordinate mapping $\pr_{\alpha} \circ \Phi : X \rightarrow Y_{\alpha}$ is continuous for every $\alpha$.

It is easy to see that coordinate mappings $\pr_{1} \circ G : (x, t) \mapsto x$ and $\pr_{2} \circ G(x, t) = t \beta(x) + (1-t) \alpha(x)$, $(x, t) \in [a, b] \times I$, are continuous since they both can be represented as compositions of continuous mappings. Therefore $G$ is also continuous.

The mapping $G$ is injective. It transforms linearly every segment $\{x\} \times I$ onto a segment $\{x\} \times [\alpha(x), \beta(x)]$. It is clear that the subspace $K$ of the plane $\rr^{2}$ is Hausdorff and $G([a, b] \times I) = K$. The space $[a, b] \times I$ is compact, therefore $G$ is homeomorphism onto its image $K$.
\end{proof}

\begin{proof}[Proof of lemma~\ref{lemma_f_4}]
Let us consider a set
\[
K = \bigcup_{c \in I} \{ (c, \tau) \,|\, \tau \in [0, \mu_{\varphi(c)}] \} \,,
\]
and a mapping $\Psi = \varphi \times Id : K \rightarrow R \subset \cM \times \rr$,
\[
\Psi (c, \tau) = (\varphi(c), \tau), \quad (c, \tau) \in K \,.
\]
It is clear that this mapping is continuous since both projections $pr_{1} = \varphi$ and $pr_{2} = Id$ are continuous.

We consider also a continuous mapping $\theta = q \circ \Psi : K \rightarrow \rr^{2}$, $\theta(c, \tau) = r_{\varphi(c)}(\tau)$, $(c, \tau) \in K$. Obviously, following equalities hold true
\[
\theta(\{c\} \times [0, \mu_{\varphi(c)}]) = r_{\varphi(c)}([0, \mu_{\varphi(c)}]) = \varphi(c)(I) \,.
\]

We denote $\alpha(t) = 0$, $\beta(t) = \mu_{\varphi(t)}$, $t \in I$. It is known (see~\cite{Morse_36}) that a function which associates to a continuous curve $\lambda$ its $\mu$-length $\mu_{\lambda}$ is continuous on the space $\cM$, therefore functions $\alpha$ and $\beta$ are continuous. Moreover, $\alpha(t) < \beta(t)$ for every $t \in I$ by condition of Lemma. We apply Proposition~\ref{prop_f_1} to $K$ and get a homeomorphism $G : I^{2} \rightarrow K$, $G(t, \tau) = (t, \mu_{\varphi(t)} \cdot \tau)$, $(t, \tau) \in I^{2}$ such that $G(\{t\} \times I) = \{t\} \times [0, \mu_{\varphi(c)}]$ for all $t \in I$.

Let us consider also a homeomorphism $T : I^{2} \rightarrow I^{2}$, $T(x, y) = (y, x)$, $(x, y) \in I^{2}$ and a continuous mapping $\Phi = \theta \circ G \circ T : I^{2} \rightarrow \rr^{2}$,
\[
\Phi (\tau, t) = \theta \circ G (t, \tau) = \theta (t, \mu_{\varphi(t)} \cdot \tau) = r_{\varphi(t)}(\mu_{\varphi(t)} \cdot \tau) \,, \quad (\tau, t) \in I^{2} \,.
\]
This mapping complies with the equalities
\[
\Phi(I \times \{t\}) = \theta \circ G (\{t\} \times I) = \theta (\{t\} \times [0, \mu_{\varphi(t)}]) = \varphi(t)(I) \,.
\]
Lemma is proved.
\end{proof}

%% file: foliation_rectification.tex
\section{Rectification of foliations on disk which are induced by regular functions.}

What we said above allows us to prove the following theorem.

\begin{theorem}\label{theorem_f_1}
Let a continuous function $f : I^{2} \rightarrow \rr$ complies with the following conditions:
\begin{itemize}
	\item $f([0, 1] \times \{0\}) = 0$, $f([0, 1] \times \{1\}) = 1$;
	\item every point of the set $\mathring{I}^{2}$ is a regular point of $f$;
	\item all points of a set $\{0, 1\} \times (0, 1)$ are regular boundary points of $f$;
	\item through any point of a subset $\Gamma$ dense in $(0, 1) \times \{0, 1\}$ passes a $U$-trajectory.
\end{itemize}

Then there exists a homeomorphism $H_{f} : I^{2} \rightarrow I^{2}$ such that $H_{f}(z) = z$ for all $z \in I \times \{0, 1\}$ and $f \circ H_{f} (x, y) = y$ for every $(x, y) \in I^{2}$.
\end{theorem}

\begin{proof}
For every value $c \in I$ of the function $f$ we fix a parameterization $\zeta_{c} : I \rightarrow \rr^2$ of the level curve $f^{-1}(c)$ to satisfy equalities $\zeta_{c}(0) \in \{0\} \times I$ (see Corollary~\ref{nas_prop_f_0}).

We consider a mapping $\varphi : I \rightarrow \cM$, $\varphi(c) = \zeta_{c}$, $c \in I$. From Lemma~\ref{lemma_f_3} it follows that this map is continuous. Moreover, it is known (see~\cite{Morse_36}) that for every continuous curve $\lambda$ an inequality $\mu_{\lambda} \geq (\diam \lambda(I))/2$ holds true. Therefore $\mu_{\zeta_{c}} > 0$ for every $c \in I$ and $\varphi$ complies with the condition of Lemma~\ref{lemma_f_4}.

Let $\Phi : I^{2} \rightarrow \rr^{2}$, $\Phi(\tau, t) = r_{\zeta_{t}}(\mu_{\zeta_{t}} \cdot \tau)$, $(\tau, t) \in I^{2}$ is a continuous mapping from Lemma~\ref{lemma_f_4}. Then
\[
\Phi(I^{2}) = \bigcup_{c \in I} \Phi(I \times \{c\}) = \bigcup_{c \in I} \zeta_{c}(I) = \bigcup_{c \in I} f^{-1}(c) = I^{2} \,.
\]

For every simple continuous curve $\zeta_{c}$, $c \in I$, its $\mu$-parametrization $r_{\zeta_{c}}$ is a simple continuous curve, so for every $c \in I$ a mapping
\[
\restrict{\Phi}{I \times \{c\}} : I \times \{c\} \rightarrow \zeta_{c}(I)
\]
is injective. More than that, when $c \neq d$ we obviously have
\[
\Phi(I \times \{c\}) \cap \Phi(I \times \{d\}) = \zeta_{c}(I) \cap \zeta_{d}(I) = f^{-1}(c) \cap f^{-1}(d) = \emptyset \,.
\]
Therefore $\Phi$ is injective mapping. It is known that a continuous injective mapping of compact into a Hausdorff space is a homeomorphism onto its image, hence $\Phi : I^{2} \rightarrow I^{2}$ is homeomorphism.

Let us denote $H_{f} = \Phi$. It is obvious that $H_{f}(x, y) \in \zeta_{y}(I)$ and $\zeta_{y}(I) = f^{-1}(y)$, so $f \circ H_{f}(x, y) = y$ for all $(x, y) \in I^{2}$.

It is straightforward that if a support of a continuous curve $\lambda : I \rightarrow \rr^2$ is a linear segment of the length $s$, then $\mu_{n}(\lambda) = s/n$, $n \in \nn$,
\[
\mu_{\lambda} = \sum_{n \in \nn} \frac{s}{n2^n} = s \cdot S \,, \quad S = \sum_{n \in \nn} \frac{1}{n 2^n} \,,
\]
and $r_{\lambda} : [0, \mu_{\lambda}] \rightarrow \lambda(I)$ maps a segment $[0, \mu_{\lambda}] = [0, s \cdot S]$ linearly onto $\lambda(I)$.

Consequently
\begin{gather*}
H_f(\tau, 0) = \Phi(\tau, 0) = r_{\zeta_{0}}(\mu_{\zeta_{0}} \cdot \tau) = (\tau, 0) \,, \\
H_f(\tau, 1) = \Phi(\tau, 1) = r_{\zeta_{1}}(\mu_{\zeta_{1}} \cdot \tau) = (\tau, 1) \,, \quad \tau \in I \,.
\end{gather*}
So, $H_{f}(z) = z$ for all $z \in I \times \{0, 1\}$.
\end{proof}

\begin{nas}\label{nas_f_1}
Let a continuous function $f : I^{2} \rightarrow \rr$ complies with all conditions of Theorem~\ref{theorem_f_1} except the first one, instead of which the following condition is fulfilled:
\begin{itemize}
	\item $f([0, 1] \times \{0\}) = f_{0}$, $f([0, 1] \times \{1\}) = f_{1}$ for certain $f_{0}$, $f_{1} \in \rr$, $f_{0} \neq f_{1}$.
\end{itemize}

Then there exists a homeomorphism $H_{f} : I^{2} \rightarrow I^{2}$ such that $H_{f}(z) = z$ for all $z \in I \times \{0, 1\}$ and $f \circ H_{f} (x, y) = (1-y) f_{0} + y f_{1}$ for every $(x, y) \in I^{2}$.
\end{nas}

\begin{proof}
Let us consider a homeomorphism $h : \rr \rightarrow \rr$,
\[
h(t) = \frac{t-f_0}{f_1-f_0} \,.
\]
An inverse mapping $h^{-1} : \rr \rightarrow \rr$ is given by a relation $h^{-1}(\tau) = (f_1-f_0) \tau + f_0 = \tau f_1 + (1-\tau) f_0$.

It is clear that a function $\tilde{f} = h \circ f$ satisfies condition of Theorem~\ref{theorem_f_1}. Therefore there exists a homeomorphism $H_{\tilde{f}} : I^2 \rightarrow I^2$ which fixes top and bottom sides of the square and such that $\tilde{f} \circ H_{\tilde{f}}(x, y) = y$, $(x, y) \in I^2$. Then $f \circ H_{\tilde{f}} (x, y) = h^{-1} \circ \tilde{f} \circ H_{\tilde{f}} (x, y) = h^{-1}(y) = y f_1 + (1-y) f_0$, $(x, y) \in I^2$ and the mapping $H_{f} = H_{\tilde{f}}$ complies with the condition of Corollary.
\end{proof}

We shall need the following lemma.

\begin{lem}\label{lemma_f_5}
Let $[a, b] \in \rr$ and $\alpha$, $\beta : [a, b] \rightarrow \rr$ are such continuous functions that $\alpha(t) < \beta(t)$ for every $t \in [a, b]$. Let
\begin{gather*}
K = \{ (x, y) \in \rr^{2} \,|\, y \in [a, b], x \in [\alpha(y), \beta(y)] \} \,, \\
\mathring{K} = \{ (x, y) \in \rr^{2} \,|\, y \in (a, b), x \in (\alpha(y), \beta(y)) \} \,.
\end{gather*}
Suppose that a continuous function $f : K \rightarrow \rr$ satisfies following conditions:
\begin{itemize}
	\item $f([\alpha(a), \beta(a)] \times \{a\}) = f_0$, $f([\alpha(b), \beta(b)] \times \{b\}) = f_1$ for some $f_0 \neq f_1$;
	\item every point of the set $\mathring{K}$ is regular point of $f$;
	\item all points of the set $\bigl\{ (x, y) \,|\, y \in (a, b), x \in \{\alpha(y), \beta(y)\} \bigr\}$ are regular boundary points of $f$;
	\item through any point of a set $\Gamma$ dense in  $\bigl((\alpha(a), \beta(a)) \times \{a\}\bigr) \cup \bigl((\alpha(b), \beta(b)) \times \{b\}\bigr)$ passes an $U$-trajectory.
\end{itemize}

Then there exists a homeomorphism $H_{f} : K \rightarrow K$ such that $H_{f}(z) = z$ for all $z \in \bigl([\alpha(a), \beta(a)] \times \{a\}\bigr) \cup \bigl([\alpha(b), \beta(b)] \times \{b\}\bigr)$ and
\[
f \circ H_{f} (x, y) = \bigl((b-y)f_0 + (y-a)f_1\bigr)/(b-a)
\]
for every $(x, y) \in K$.
\end{lem}

\begin{proof}
Let $T : I^2 \rightarrow I^2$, $T(x, y) = (y, x)$, $(x, y) \in \rr^2$. Let us designate by $\pr_1$, $\pr_2 : \rr^2 \rightarrow \rr$ projections on corresponding coordinates.

We consider a set $K^T = \{ (x, y) \,|\, T(x, y) \in K \}$ and use Proposition~\ref{prop_f_1} to map it onto a rectangle $[a, b] \times I$ with the help of a homeomorphism $G$. Note that on construction $\pr_1 \circ G(x, y) = x$, $(x, y) \in K^T$.

Let us examine a homeomorphism $\hat{G} = T \circ G \circ T : K \rightarrow I \times [a, b]$ and a linear homeomorphism $L : I \times [a, b] \rightarrow I^2$, $L(x, y) = (x, (y-a)/(b-a))$, $(x, y) \in I \times [a, b]$. Denote $F = L \circ \hat{G} : K \rightarrow I^2$.
Clearly $F$ is homeomorphism. It is easy to see that $\pr_2 \circ \hat{G}(x, y) = y$, $(x, y) \in K$, hence $\pr_{2} \circ F(x, y) = (y-a)/(b-a)$ for every $(x, y) \in K$.

Consider a continuous function $\hat{f} = f \circ F^{-1} : I^2 \rightarrow \rr$. A straightforward verification shows that $\hat{f}$ complies with condition of Corollary~\ref{nas_f_1}, therefore there exists a homeomorphism $H_{\hat{f}} : I^2 \rightarrow I^2$ which is identity on the set $I \times \{0, 1\}$ and such that $\hat{f} \circ H_{\hat{f}}(x, y) = f_1 y + f_0 (1-y)$ for all $(x, y) \in I^2$.

Let us denote $H_{f} = F^{-1} \circ H_{\hat{f}} \circ F : K \rightarrow K$. It is easy to see that
\[
F\Bigl(\bigl([\alpha(a), \beta(a)] \times \{a\}\bigr) \cup \bigl([\alpha(b), \beta(b)] \times \{b\}\bigr)\Bigr) = I \times \{0, 1\} \,,
\]
therefore form Corollary~\ref{nas_f_1} it follows that $H_{f}(z) = F^{-1} \circ H_{\hat{f}} \circ F(z) = F^{-1} \circ F(z) = z$ for every $z \in \bigl([\alpha(a), \beta(a)] \times \{a\}\bigr) \cup \bigl([\alpha(b), \beta(b)] \times \{b\}\bigr)$.

Moreover, for every $(x, y) \in K$ we have $f \circ H_{f}(x, y) = f \circ F^{-1} \circ H_{\hat{f}} \circ F(x, y) = \hat{f} \circ H_{\hat{f}} \circ F(x, y) = f_1 \tau + f_0 (1-\tau)$, where $\tau = \pr_2 \circ F(x, y) = (y-a)/(b-a)$. Taking into account an equality $1 - \tau = (b-y)/(b-a)$, finally we obtain
\[
f \circ H_f (x, y) = \frac{(y-a)f_1 + (b-y)f_2}{b-a} \,, \quad (x, y) \in K \,.
\]
Q. E. D.
\end{proof}

Let us introduce following notation: $a_{-} = (-1, 0)$, $a_{+} = (1, 0)$,
\begin{gather*}
\Cl{D}^{2}_{+} = \{ z \;|\; |z| \leq 1 \mbox{ and } Im z \geq 0 \} \,,\quad
\mathring{D}^{2}_{+} = \{ z \;|\; |z| < 1 \mbox{ and } Im z > 0 \} \,, \\
S_{+} = \{ z \;|\; |z| = 1 \mbox{ and } Im z \geq 0 \} \,, \quad
\mathring{S}_{+} = S_{+} \setminus \{ a_{-}, a_{+} \} \,.
\end{gather*}

\begin{theorem}\label{theorem_f_2}
Let a continuous function $f : \Cl{D}^{2}_{+} \rightarrow \rr$ complies with conditions:
\begin{itemize}
	\item every point of the set $\mathring{D}^{2}_{+}$ is a regular point of $f$;
	\item a certain point $v \in \mathring{S}_{+}$ is local maximum of $f$, all the rest points of $\mathring{S}_{+}$ are regular boundary points of $f$;
	\item $f([-1, 1] \times \{0\}) = 0$, $f(v) = 1$;
	\item through every point of a set $\Gamma$ which is dense in $(0, 1) \times \{0, 1\}$ passes an $U$-trajectory.
\end{itemize}

Then there exists a homeomorphism $H_{f} : \Cl{D}^{2}_{+} \rightarrow \Cl{D}^{2}_{+}$ such that $H_{f}(z) = z$ for all $z \in [-1, 1] \times \{0\}$ and $f \circ H_{f} (x, y) = y$ for every $(x, y) \in \Cl{D}^{2}_{+} \subset \rr^2$.
\end{theorem}

\begin{proof}
We designate by $\gamma_{-}$ and $\gamma_{+}$ close arcs which are contained in $S_{+}$ and join with $v$ points $a_{-}$ and $a_{+}$ respectively. Let $\mathring{\gamma}_{-} = \gamma_{-} \setminus \{a_{-}, v\}$ and $\mathring{\gamma}_{+} = \gamma_{+} \setminus \{a_{+}, v\}$ are corresponding open arcs. It is clear that on each of the arcs $\gamma_{-}$ and $\gamma_{+}$ function $f$ changes strictly monotonously from 0 to 1.

Similarly to Proposition~\ref{prop_f_0} we prove that $f$ is weakly regular on $\Cl{D}^{2}_{+}$. Like in Corollary~\ref{nas_prop_f_0} from this follows that $f(z) \in (0, 1)$ for all $z \in \Cl{D}^{2}_{+} \setminus \bigl(([-1, 1] \times \{0\}) \cup \{v\}\bigr)$ and for every $c \in (0, 1)$ a level set $f^{-1}(c)$ is a support of a simple continuous curve $\zeta_c : I \rightarrow \Cl{D}^{2}_{+}$, with $\zeta_c(0) \in \mathring{\gamma}_{-}$, $\zeta_c(1) \in \mathring{\gamma}_{+}$ and $\zeta_c(t) \in \mathring{D}^{2}_{+}$ when $t \in (0, 1)$.

We apply Proposition~\ref{prop_func_on_compact} to a level set $f^{-1}(1) = \{v\}$ and find an increasing sequence of numbers $0 = c_0 < c_1 < c_2 < \ldots < 1$, $\lim_{k \rightarrow \infty} c_k = 1$, which satisfies the following requirement: $f^{-1}(c) \subset U_{1/k}(v)$ for all $c \geq c_k$, $k \in \nn$. Here $U_{\varepsilon}(v) = \{z \in \Cl{D}^{2}_{+} \,|\, \dist(z, v) < \varepsilon \}$ is a $\varepsilon$-neighbourhood of $v$.

Let $\tilde{\zeta}_k = \zeta_{c_k} : I \rightarrow \Cl{D}^{2}_{+}$ be simple continuous curves with supports $f^{-1}(c_k)$, $k \in \nn$. Let also $\tilde{\zeta}_0 : I \rightarrow f^{-1}(0) = [-1, 1] \times \{0\} \subset \Cl{D}^{2}_{+}$, $f(t) = (2t-1, 0)$. We denote $a_{-}^k = \tilde{\zeta}_k(0) \in \mathring{\gamma}_{-}$, $a_{+}^k = \tilde{\zeta}_k(1) \in \mathring{\gamma}_{+}$ (see above), $a_{-}^0 = a_{-}$, $a_{+}^0 = a_{+}$.

Let $\gamma_{-}^k : I \rightarrow \gamma_{-}$, $k \in \nn$, be simple continuous curves such that $\gamma_{-}^k(0) = a_{-}^k$, $\gamma_{-}^k(1) = a_{-}^{k+1}$. By analogy we fix simple continuous curves $\gamma_{+}^k : I \rightarrow \gamma_{+}$ such that $\gamma_{+}^k(0) = a_{+}^k$, $\gamma_{+}^k(1) = a_{+}^{k+1}$.

We also designate $b_{-}^{k} = (-\sqrt{1-c_k^2}, c_k)$, $b_{+}^{k} = (\sqrt{1-c_k^2}, c_k) \in S_{+}$, $k \geq 0$.

For every $k \geq 0$ we fix three continuous injective mappings $\varphi_k : \tilde{\zeta}_k(I) \rightarrow \left[-\sqrt{1-c_k^2}, \sqrt{1-c_k^2}\right] \times \{c_k\}$, $\psi_{-}^k : \gamma_{-}^k(I) \rightarrow S_{+}$ and $\psi_{+}^k : \gamma_{+}^k(I) \rightarrow S_{+}$, which satisfy requirements: $\varphi_k(0) = \psi_{-}^k(0) = b_{-}^{k}$, $\varphi_k(1) = \psi_{+}^k(0) = b_{+}^{k}$, $\psi_{-}^k(1) = b_{-}^{k+1}$, $\psi_{+}^k(1) = b_{+}^{k+1}$. We can regard that $\varphi_0 = id : [-1, 1] \times \{0\} \rightarrow [-1, 1] \times \{0\}$ is an identity mapping.

Let us consider following simple continuous curves
\begin{gather*}
\xi_k = \varphi_k \circ \tilde{\zeta}_k : I \rightarrow \left[-\sqrt{1-c_k^2}, \sqrt{1-c_k^2}\right] \times \{c_k\} \subset \Cl{D}^{2}_{+} \,, \\
\eta_{-}^k = \psi_{-}^k \circ \gamma_{-}^{k} \,, \; \eta_{+}^k = \psi_{+}^k \circ \gamma_{+}^{k} : I \rightarrow S_{+} \,,
\quad k \geq 0 \,.
\end{gather*}
Let $J_k$ be a curvilinear rectangle bounded by curves $\gamma_{-}^{k}$, $\tilde{\zeta}_k$, $\gamma_{+}^{k}$ and $\tilde{\zeta}_{k+1}$, and $I_k = \left\{ (x, y) \,|\, y \in [c_k, c_{k+1}], x \in \left[-\sqrt{1-y^2}, \sqrt{1-y^2}\right] \right\}$ be a curvilinear rectangle bounded by curves $\eta_{-}^k$, $\xi_k$, $\eta_{+}^k$ and $\xi_{k+1}$. It is straightforward that the mappings $\psi_{-}^k$, $\varphi_k$, $\psi_{+}^k$ and $\varphi_{k+1}$ induce a homeomorphism $\Phi_{k}^0 : \partial J_k \rightarrow \partial I_k$ of a boundary $\partial J_k$ of the set $J_k$ onto a boundary $\partial I_k$ of $I_k$, moreover on the set $\tilde{\zeta}_k(I) = \partial J_{k-1} \cap \partial J_k$ mappings $\Phi_{k-1}^0$ and $\Phi_{k}^0$ coincide for every $k \in \nn$.

\begin{figure}
\centerline{\includegraphics{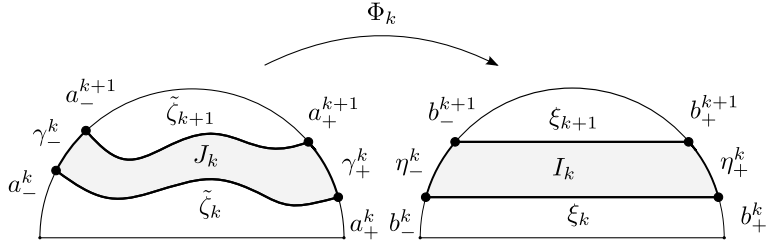}}
\caption{A homeomorphism $\Phi_k : J_k \rightarrow I_k$}\label{Fig_half_disks}
\end{figure}

We use theorem of Shoenflies (see~\cite{Z-F-C, Newman}) and for every $k \geq 0$ continue the mapping $\Phi_k^0$ to a homeomorphism $\Phi_k : J_k \rightarrow I_k$ (see Figure~\ref{Fig_half_disks}). Remark that by construction homeomorphisms $\Phi_{k-1}$ and $\Phi_k$ coincide on a set $\tilde{\zeta}_k(I) = J_{k-1} \cap J_k$ for all $k \in \nn$.

For every $k \geq 0$ we consider a function $f \circ \Phi_k^{-1} : I_{k} \rightarrow \rr$. A straightforward verification shows that this functions complies with the condition of Lemma~\ref{lemma_f_5} with $f_0 = c_k$ and $f_1 = c_{k+1}$. Therefore there exists a homeomorphism $H_k : I_k \rightarrow I_k$ which is an identity on a set $\xi_k(I) \cup \xi_{k+1}(I)$ and such that
\[
f \circ \Phi_k^{-1} \circ H_k(x, y) = \frac{(c_{k+1}-y)c_k + (y-c_k)c_{k+1}}{c_{k+1}-c_k} = y \,, \quad y \in I_k \,.
\]

It is obvious that by construction homeomorphisms $\Phi_{k-1}^{-1} \circ H_{k-1}$ and $\Phi_k^{-1} \circ H_k$ coinside on the set $\xi_k(I) = I_{k-1} \cap I_k$ for every $k \in \nn$. Therefore we can define a mapping $H_f : \Cl{D}^{2}_{+} \rightarrow \Cl{D}^{2}_{+}$,
\[
H_f(x, y) =
\begin{cases}
\Phi_k^{-1} \circ H_k(x, y) \,, & \mbox{if } y \in [c_k, c_{k+1}] \,; \\
v \,, & \mbox{if } (x, y) = (0, 1) \,,
\end{cases}
\]
and by construction it satisfies the relation $f \circ H_f(x, y) = y$, $(x, y) \in \Cl{D}^{2}_{+}$.

It is easy to see that this mapping is bijective. Moreover $H_f(z) = \varphi_0^{-1}(z) = z$ when $z \in [-1, 1] \times \{0\}$.
The set $\Cl{D}^{2}_{+}$ is compact, so for completion of the proof it is sufficient to verify continuity of $H_f$.

Let us consider the set $\tilde{D}_{+} = \Cl{D}^{2}_{+} \setminus \{(0, 1)\}$ and its covering $\{I_k\}_{k \geq 0}$. This covering is locally finite and close, so it is fundamental (see~\cite{RF}). Moreover by construction all mappings $H_f|_{I_k} = \Phi_k^{-1} \circ H_k$ are continuous. Consequently, the mapping $H_f$ is also continuous on $\tilde{D}_{+}$.

In order to prove the continuity of $H_f$ in the point $(0, 1)$ we observe that a family of sets
\[
W_k = \{(x, y) \in \Cl{D}^{2}_{+} \,|\, y > c_k\} = \{(x, y) \in \Cl{D}^{2}_{+} \,|\, f \circ H_f(x, y) > c_k\} \,, \quad k \in \nn \,,
\]
forms the base of neighbourhoods of $(0, 1)$. We sign
\[
V_k = H_f(W_k) = \{(x, y) \in \Cl{D}^{2}_{+} \,|\, f(x, y) > c_k\} \,, \quad k \in \nn \,.
\]

According to the choice of numbers $\{c_k\}_{k \geq 0}$ for every $c \geq c_{k}$ the inequality $f^{-1}(c) \subset U_{1/k}(v)$ holds true, $k \in \nn$. So
\[
V_k \subseteq U_{1/k}(v) \,, \quad k \in \nn \,.
\]
A family of sets $\{U_{1/k}(v)\}_{k \in \nn}$ forms the base of neighbourhoods of $v = H_f(0, 1)$ and for every $k \in \nn$ the inequality $H_f^{-1}(U_{1/k}(v)) \supseteq W_k = H_f^{-1}(V_k)$ is valid. Consequently, the mapping $H_f$ is continuous in $(0, 1)$, and hence it is continuous on $\Cl{D}^{2}_{+}$.

Q. E. D.
\end{proof}

Similarly to~\ref{nas_f_1} the following statement is proved.

\begin{nas}\label{nas_f_2}
Assume that a continuous function $f : \Cl{D}^{2}_{+} \rightarrow \rr$ complies with the requirements:
\begin{itemize}
	\item every point of the set $\mathring{D}^{2}_{+}$ is a regular point of $f$;
	\item a certain point $v \in \mathring{S}_{+}$ is a local extremum of  $f$, all the rest points of $\mathring{S}_{+}$ are regular boundary points of $f$;
	\item $f([-1, 1] \times \{0\}) = f_0$, $f(v) = f_1$ for some $f_0$, $f_1 \in \rr$, $f_0 \neq f_1$;
	\item through every point of a set $\Gamma$, which is dense in $(0, 1) \times \{0, 1\}$, passes an $U$-trajectory.
\end{itemize}

Then there exists a homeomorphism $H_{f} : \Cl{D}^{2}_{+} \rightarrow \Cl{D}^{2}_{+}$ such that $H_{f}(z) = z$ for all $z \in [-1, 1] \times \{0\}$ and $f \circ H_{f} (x, y) = (1-y) f_{0} + y f_{1}$ for every $(x, y) \in \Cl{D}^{2}_{+}$.
\end{nas}

\begin{nas}\label{nas_f_3}
Let a continuous function $f : D^2 \rightarrow \rr$ satisfies the conditions:
\begin{itemize}
	\item every point of the set $\Int{D^2}$ is a regular point of $f$;
	\item certain points $v_{+}, v_{-} \in S = \Fr{D^2}$ are local maximum and minimum of $f$ respectively; all other points of $S$ are regular boundary points of $f$;
\end{itemize}

Then there exists a homeomorphism $H_{f} : D^2 \rightarrow D^2$ such that
\[
f \circ H_{f} (x, y) = \frac{(1-y) f(v_{-}) + (1+y) f(v_{+})}{2} \,, \quad (x, y) \in D^2 \,.
\]
\end{nas}

\begin{proof}
Similarly to Proposition~\ref{prop_f_0} it is proved that the function $f$ is weakly regular on $D^2$.

Let $\gamma_1$, $\gamma_2 = \{v_{-}\}$, $\gamma_3$ and $\gamma_4 = \{v_{+}\}$ be the arcs from Definition~\ref{ozn_weak_regular}.
By analogy with Corollary~\ref{nas_prop_f_0} it is proved that $f(z) \in (f(v_{-}), f(v_{+}))$ for all $z \in D^2 \setminus \{v_{+}, v_{-}\}$, and also for each $c \in (f(v_{-}), f(v_{+}))$ a level set $f^{-1}(c)$ is a support of a simple continuous curve $\zeta_c : I \rightarrow D^2$, moreover $\zeta_c(0) \in \mathring{\gamma}_1$, $\zeta_c(1) \in \mathring{\gamma}_3$ and $\zeta_c(t) \in \Int D^2$ for $t \in (0, 1)$.

Let $c_0 = (f(v_{-}) + f(v_{+}))/2$. It is straightforward that a set $f^{-1}(c_0)$ divides disk $D^2$ into two parts, one of which contains the point $v_{-}$, the other contains $v_{+}$. We denote closures of connected components of $D^2 \setminus f^{-1}(c_0)$ by $D_{-}$ and $D_{+}$ respectively. Each of these sets is homeomorphic to closed disk and correlations $D_{-} = \{z \in D^2 \,|\, f(z) \leq c_0\}$, $D_{+} = \{z \in D^2 \,|\, f(z) \geq c_0\}$, $v_{-} \in D_{-}$, $v_{+} \in D_{+}$, $D_{-} \cap D_{+} = f^{-1}(c_0)$ are fulfilled.

The set $f^{-1}(c_0)$ is the support of a simple continuous curve $\zeta : I \rightarrow D^2$ (see above). For every $t \in (0, 1)$ a point $\zeta(t)$ is a regular point of $f$, therefore through this point passes a $U$-trajectory and it is divided by the point $\zeta(t)$ into two arcs, one of which is contained in $D_{-}$, the other belongs to $D_{+}$. Consequently, in each of the sets $D_{-}$ and $D_{+}$ through the point $\zeta(t)$ passes a $U$-trajectory, so we can take advantage of Corollary~\ref{nas_f_2} and by means of a straightforward verification we establish validity of the following claims:
\begin{itemize}
	\item there exists such a homeomorphism $H_{-} : D_{-} \rightarrow \Cl{D}^{2}_{+}$ that $H_{-} \circ \zeta(t) = (2t-1, 0)$, $t \in I$ and
\begin{gather*}
f \circ H_{-}^{-1}(x, y) = (1-y) c_0 + y f(v_{-}) = \\
= \frac{(1-y)(f(v_{-})+f(v_{+}))}{2} + y f(v_{-}) = \frac{(1+y)f(v_{-})}{2} + \frac{(1-y)f(v_{+})}{2} \,;
\end{gather*}
	\item there exists a homeomorphism $H_{+} : D_{+} \rightarrow \Cl{D}^{2}_{+}$ which complies with the equalities $H_{+} \circ \zeta(t) = (2t-1, 0)$, $t \in I$ and
\[
f \circ H_{+}^{-1}(x, y) = (1-y) c_0 + y f(v_{+}) = \frac{(1-y)f(v_{-})}{2} + \frac{(1+y)f(v_{+})}{2} \,.
\]
\end{itemize}

Let us consider a set $\Cl{D}^{2}_{-} = \{ (x, y) \in D^2 \,|\, y \leq 0 \}$ and a homeomorphism $Inv : \Cl{D}^{2}_{+} \rightarrow \Cl{D}^{2}_{-}$, $Inv(x, y) = (x, -y)$. A mapping $\hat{H}_{-} = Inv \circ H_{-} : D_{-} \rightarrow \Cl{D}^{2}_{-}$ is obviously a homeomorphism and is compliant with the equalities $\hat{H}_{-} \circ \zeta(t) = (2t-1, 0)$, $t \in I$ and
\[
f \circ \hat{H}_{-}^{-1}(x, y) = \frac{(1+(-y))f(v_{-})}{2} + \frac{(1-(-y))f(v_{+})}{2} = \frac{(1-y)f(v_{-})}{2} + \frac{(1+y)f(v_{+})}{2} \,.
\]

From the above it easily follows that a mapping $H_f : D^2 \rightarrow D^2$,
\[
H_f(x, y) =
\begin{cases}
\hat{H}_{-}(x, y) \,, & \mbox{if } (x, y) \in D_{-} \,, \\
H_{+}(x, y) \,, & \mbox{if 	} (x, y) \in D_{+} \,,
\end{cases}
\]
is a homeomorphism and satisfies the hypothesis of Corollary.
\end{proof}

Let us summarize claims proved in this subsection.

Taking into account Lemma~\ref{lem_number_of_arks} we can give the following definition.
\begin{ozn}\label{ozn_regular_on_disk}
Let $f$ be a weakly regular function on the disk $D$, let $\gamma_1, \ldots, \gamma_4$ be arcs from Definition~\ref{ozn_weak_regular}. If through every point of a set $\Gamma$ which is dense in $\mathring{\gamma}_2 \cup \mathring{\gamma}_4$ passes a $U$-trajectory, then the function $f$ is called \emph{regular on $D$}.
\end{ozn}

\begin{theorem}\label{theorem_f_3}
Let $f$ be a regular function on the disk $D$, let $\gamma_1, \ldots, \gamma_4$ be arcs from Definition~\ref{ozn_weak_regular}. Let $D' = I^2$ if $\mathring{\gamma}_2 \neq \emptyset$ and $\mathring{\gamma}_4 \neq \emptyset$; $D' = D^2$ if $\mathring{\gamma}_2 \cup \mathring{\gamma}_4 = \emptyset$; $D' = \Cl{D}^{2}_{+}$ if exactly one from the sets $\mathring{\gamma}_2$ or $\mathring{\gamma}_4$ is empty.

Let $\phi : \Fr{D} \rightarrow \Fr{D'}$ be a homeomorphism such that $\phi(K) = K'$, where
\begin{gather*}
K = f^{-1}\bigl(\min_{z \in D}(f(z)) \cup \max_{z \in D}(f(z))\bigr) \,,\\
K' = \Bigl\{ (x, y) \in D' \,\bigl|\, y \in  \bigl\{\min_{(x, y) \in D'}(y), \max_{(x, y) \in D'}(y) \bigr\}\bigr.\Bigr\} \,.
\end{gather*}

Then there exists a homeomorphism $H_f$ of $D$ onto $D'$ such that $H_f|_K = \phi$ and $f \circ H_f^{-1}(x, y) = ay + b$, $(x, y) \in D'$ for certain $a, b \in \rr$, $a \neq 0$.
\end{theorem}

\begin{theorem}
Let $f$ and $g$ be regular functions on a closed 2-disk $D$.

Every homeomorphism $\varphi_0 : \partial D \rightarrow \partial D$ of the frontier $\partial D$ of $D$ which complies with the equality $g \circ \varphi_0 = f$ can be extended to a homeomorphism $\varphi : D \rightarrow D$ which satisfies the equality $g \circ \varphi = f$.
\end{theorem}

\begin{proof}
This statement is a straightforward corollary from Theorem~\ref{theorem_f_3}.
\end{proof}


\begin{remk}
Everything said here about $\mu$-length of a curve and about Frechet distance between curves remains true for continuous curves in every separable metric space (see~\cite{Morse_36}). In particular, proof of Lemma~\ref{lemma_f_4} is literally transferred to that case.
\end{remk}

\begin{remk}
In order to prove Theorem~\ref{theorem_f_1} we used techniques analogous to the one of~\cite{Toki_51}.
\end{remk}

%% file: Main.bbl
\begin{thebibliography}{99}

\bibitem{MrJ}
Jenkins J. A, Morse.M.
\emph{Contour equivalent pseudoharmonic functions and pseudoconjugates},
Amer. J. Math., vol. \textbf{74} (1952), P.~23--51

\bibitem{Yu}
Yurchuk I.A.
\emph{Topological equivalence of functions from $F(D^{2})$ class},
Zbirn. nauk. prac Inst. math Ukr., vol. \textbf{3}, N 3 (2006), P.~474--486 (in Ukrainian)

\bibitem{Morse}
Morse M.
\emph{Topological methods in the theory of functions of a complex variable},
Princeton, 1947

\bibitem{Toki_51}
T\^{o}ki Y.
\emph{A topological characterization of pseudo-harmonic functions},
Osaka Math. Journ., vol. 3, N 1 (1951), P.~101--122

\bibitem{Whitney_33}
Whitney H.
\emph{Regular families of curves},
Annals of Math., vol. 34 (1933), P.~244--270

\bibitem{Morse_36}
Morse M.
\emph{A special parameterization of curves},
Bull. Amer. Math. Soc., vol. 42 (1936), P.~915--922.

\bibitem{RF}
Fuks D. B., Rokhlin V. A.
\emph{Beginner's course in topology. Geometric chapters}.
Translated from the Russian by A. Iacob.
Universitext. Springer Series in Soviet Mathematics. Springer-Verlag, Berlin, 1984. xi+519 pp.

\bibitem{Z-F-C}
Zieschang H., Э. Vogt E., Coldewey H.-D.
\emph{Surfaces and planar discontinuous groups}.
Springer-verlag, 1981.

\bibitem{Newman}
Newman M. H. A.
\emph{Elements of the topology of plane sets of points},
Cambridge: Cambridge Univ. Press, 1964. 214 pp.

\end{thebibliography}
